\documentclass[11pt,a4paper,reqno]{amsart}

\usepackage{geometry}
\usepackage{amssymb}
\usepackage{times}
\usepackage{mathrsfs}
\usepackage{verbatim}

\usepackage{amsrefs,doi,color}
\newcommand\half{\frac12}

\theoremstyle{plain}
\newtheorem{theorem}{Theorem}[section]

\theoremstyle{remark}
\newtheorem{remark}[theorem]{Remark}

\newtheorem{example}[theorem]{Example}

\theoremstyle{plain}

\newtheorem{lemma}[theorem]{Lemma}
\newtheorem{proposition}[theorem]{Proposition}
\newtheorem{definition}[theorem]{Definition}
\newtheorem{hypothesis}[theorem]{Hypothesis}

\numberwithin{equation}{section}

\newcommand\curl{\operatorname{curl}}

\newcommand\id{\operatorname{id}}

\newcommand{\eps}{\varepsilon}

\def\Rnu{{\mathbb R}}

\def\im{\mbox{im\,}}
\def\liml{\lim\limits}
\def\suml{\sum\limits}
\def\supl{\sup\limits}

\def\intl{\int\limits}

\def\supl{\mathop{\sup}\limits}

\begin{document}
\title{Global solutions of the random vortex filament equation}

\author{Z. Brze\'{z}niak$^\dag$, M. Gubinelli$^\ddag$ and
M. Neklyudov$^\dag$}

\address{
$^\dag$Department of Mathematics, University of York, Heslington,
UK
\\
$^\ddag$CEREMADE \& CNRS UMR 7534, Universit\'e Paris-Dauphine,  France
}

\thanks{This research has been supported by the ANR Project ECRU (ANR-09-BLAN-0114-01/2)}

\date{\today}

\begin{abstract}
We prove the existence of a global solution for the
 filament equation with inital condition given by a geometric rough path in the sense of Lyons~\cite{[Lyons-1998]}. Our work gives a positive answer to a question left open in recent publications: Berselli and  Gubinelli \cite{[Bers+Gub_2007]}  showed the existence of a global solution for a smooth  initial condition  while  Bessaih, Gubinelli, Russo \cite{[Bess+Gub_2005]} proved the existence of a local solution for a general  initial condition given by a rough path.
\end{abstract}

\maketitle

\section{The filament equation}
In this note we prove the existence of a global solution for the following rough
 filament equation
\begin{equation}\label{eqn:RandomFilam-1}
\begin{cases}
\frac{d\gamma(t)(\xi)}{dt} = u^{\gamma(t)}(\gamma(t)(\xi)), \;\; t\in [0,\infty),\xi\in[0,1], \\
\gamma(0) = \gamma_0,
\end{cases}
\end{equation}
where the initial condition $\gamma_0:[0,1]\to\mathbb{R}^3$ is a  geometric $\nu$-rough path (for some  $\nu\in(\frac{1}{3},1)$),  see Assumption \ref{ass:RoughPathApproximability}. Here $\gamma:[0,\infty)\to \mathcal{D}_{\gamma_0}\subset \mathcal{C}$ is some trajectory in
the subset $\mathcal{D}_{\gamma_0}$ of the space $\mathcal{C}\subseteq C([0,1];\Rnu^3)$ of continuous closed curves in
$\Rnu^3$ parametrized by the interval $[0,1]$. Finally $u^\gamma$, $\gamma\in \mathcal{D}_{\gamma_0}$ is a vector field
given by
\begin{equation}
\label{eqn:RandomFilam-3}
u^\gamma(x)=\int_0^1\nabla\phi(x-\gamma(\xi))\times d\gamma(\xi),
\end{equation}
where $\phi:\Rnu^3\to\Rnu$ is a smooth function which satisfies
certain assumptions (see Hypothesis \ref{hyp:PhiAssumptions}).
The exact meaning of the line integral above and the definition of the set $\mathcal{D}_{\gamma_0}$ will be given below. Equation \eqref{eqn:RandomFilam-1} has its origin in
the theory of the three dimensional Euler equations.
It is well known that for the two dimensional Euler equations the vorticity ${\omega}=\curl {u}$ is transported along the fluid flow.
The situation changes drastically in three dimensional case. Additional ``stretching" term in the equation defining the vorticity
leads to a possibility of its blow up in finite time.
Furthermore, a result of Beale,
Kato, and Majda \cite{[BKM-1984]} suggests that a possible
singularity of the Euler equations appears when the vorticity field of the
fluid blows up. Consequently, understanding the behaviour of the vorticity of an ideal incompressible fluid is an  important problem in fluid dynamics.

The properties of the motion of the vorticity has been studied for the last 150 years starting from the works of
Helmholtz \cite{Helm1858} and Kelvin \cite{Kelvin1869}. It has been suggested by Kelvin  to use the so called Biot-Savart law
\[
{u}({x})=\int_{\mathbb{R}^3}\frac{{x}-{y}}{\vert {x}-{y}\vert^3}\times{\omega}({y})d{y},\qquad{x}\in\mathbb{R}^3
\]
where $\times$ denotes the vector product in $\mathbb{R}^3$,
combined with an assumption that the vorticity is supported by a smooth closed curve evolving in time $\gamma:\mathbb{R}_+\times[0,1]\to\mathbb{R}^3$ via a formula 
\[
\omega({x},t)=\Gamma\int_0^1\delta({x}-\gamma(t,\xi))\frac{\partial \gamma(t,\xi)}{\partial \xi}d\xi,\qquad x\in\mathbb{R}^3,t\geq 0
\]
and definition of the flow
\begin{equation}\label{equa21}
\left\{
\begin{array}{rcl}
\frac{d {X}_t({x})}{dt}&=&{u}({X}_t({x}),t),\;\;t\geq 0,\\
{X}_0({x})&=&{x}.
\end{array}
\right.
\end{equation}
to formally deduce the \emph{filament equation}
\begin{equation}\label{eqn:RandomFilamClassic}
\frac{\partial \gamma}{\partial t}(t,\xi)=-\frac{\Gamma}{4\pi}\int_0^1\frac{\gamma(t,\xi)-\gamma(t,\eta)}{\vert \gamma(t,\xi)-\gamma(t,\eta)\vert^3}\times\frac{\partial\gamma(t,\eta)}{\partial \eta}d\eta.
\end{equation}
The assumption that the vorticity is supported by some curve is coherent with numerical simulations of 3D turbulent fluids
which show that regions of large  vorticity   have a form of a
``filament", see for instance \cite{[Bell+Markus_1992]} and 
\cite{[VincMeneguzzi-1991]}.

Equation \eqref{eqn:RandomFilamClassic} has singularity when $\xi$ and $\eta$ are close to each other and the initial curve $\gamma$ is smooth.
As a consequence, the energy of the solution of this equation given by the formula
\[
E(t)=\frac{\Gamma^2}{8\pi}\int_0^1\int_0^1\frac{1}{\vert \gamma(t,\xi)-\gamma(t,\eta)\vert }\frac{\partial \gamma(t,\eta)}{\partial \eta}\cdot\frac{\partial \gamma(t,\xi)}{\partial \xi}d\xi d\eta ,\qquad t\geq 0
\]
is infinite for any smooth curve $\gamma(t,\cdot)$. In order to avoid the singularity
Rosenhead~\cite{[Rosenhead-1930]}  introduced the regularized  model
\begin{equation}\label{eqn:RandomFilamRosenhead}
\frac{\partial \gamma}{\partial t}(t,\xi)=-\frac{\Gamma}{4\pi}\int_0^1\frac{\gamma(t,\xi)-\gamma(t,\eta)}{(\vert \gamma(t,\xi)-\gamma(t,\eta)\vert^2+\mu^2)^{3/2}}\times\frac{\partial\gamma(t,\eta)}{\partial \eta}d\eta
\end{equation}
which is widely used in aereonautics to study the dynamics of trailing vortices at the tips of airplane wings.

The problem~\eqref{eqn:RandomFilam-1} has been studied  by Berselli and Bessaih~\cite{[Bers+Bess_2002]} and then by Berselli and Gubinelli~\cite{[Bers+Gub_2007]}.
It contains the equation~\eqref{eqn:RandomFilamRosenhead}  as a  particular case when, for for some $\mu>0$, 
\[
\phi({x})=\frac{\Gamma}{(\vert {x}\vert^2+\mu^2)^{\half}},\qquad {x}\in\mathbb{R}^3.
\]
Equation~\eqref{eqn:RandomFilam-1} is in fact  a nonlinear PDE for a function $\gamma: [0,\infty)\times[0,1]\to \mathbb{R}^3$. A natural setting for the  well-posedness of the  corresponding Cauchy problem is obtained by requiring  the vector field $u$ to be well defined and Lipshitz with respect to  the space variable.
To this effect the approach followed in~\cite{[Bers+Bess_2002]} is to set up the equation as an evolution problem in the Sobolev space $H^1$ of closed absolutely continuous curves in $\mathbb{R}^3$ with square integrable first derivative (with respect to the parameter). This approach implies that the vector field (once lifted to $H^1$) does not allow   strong enough estimates to  deduce the global existence.
This issue is ultimately due to the fact that in the $3$-dimensional  incompressible flows, the  vortices stretch and undergo a complex dynamics and that a priori this could lead to a explosion of the $H^1$ norm.
Such a difficulty should be compared to the more stable behavior of 2d vortex points which, under incompressible flows, are simply transported along the flow lines.
Exploiting the conservation of the kinetic energy of the flow and a control of the velocity field generated by the vortex line via the associated kinetic energy Berselli and  Gubinelli \cite{[Bers+Gub_2007]}  showed the existence of a global solution to  equation~\eqref{eqn:RandomFilam-1} with initial conditions in $H^1$.

Bessaih, Gubinelli and Russo \cite{[Bess+Gub_2005]}, partially motivated by the random filament models suggested by Gallavotti~\cite{Gallavotti} and  Chorin~\cite{Chorin}, considered the above evolution problem when the initial data  is a random closed curve. For definiteness they  considered the initial data sampled from the law of a $3$-dimensional Brownian loop (a Brownian motion starting at $0\in\mathbb{R}^3$ and conditioned to return to $0\in\mathbb{R}^3$ at time $1$).
In this case it is no more possible to set up the problem in the Sobolev space $H^1$ since Brownian trajectories almost surely do not belong to this space. A more serious problem is the meaning to give to the generalized Biot-Savart relation~\eqref{eqn:RandomFilam-3} since the line integral along a Brownian trajectory does not allow a straighforward definition. Moreover stochastic integration (\`a la It\^o or Stratonovich) does not provide a good framework to study this problem since it does not posses a natural filtration on the parameter space  stable under the time evolution. For these reasons~\cite{[Bess+Gub_2005]} identified the spaces of controlled rough paths as a natural framework to have a well posed problem.

Rough path theory has been introduced by T. J. Lyons in the seminal paper~\cite{[Lyons-1998]} (see also~\cites{[LyonsQian-2002],[LionsStFlour],[FrizVictoir]}) as a way to overcome certain difficulties of stochastic integration theories and have a robust analytical framework to solve stochastic differential equations and similar problems involving integration of non-regular vector-fields.
It turns out that rough paths theory and in particular the notion of \emph{controlled paths} introduced in~\cite{[Gubinelli-2004]} allows to give a natural interpretation to the Biot-Savart relation~\eqref{eqn:RandomFilam-3} and obtain a well-posed problem. Using this approach  Bessaih, Gubinelli, Russo \cite{[Bess+Gub_2005]}  obtained existence of a local solution to  eq.~\eqref{eqn:RandomFilam-1} when the initial data is a closed curve of
H\"{o}lder class with exponent $\nu\in ({1}/{3},1]$ (suitably lifted to rough path space).

The aim of the present paper is to extend the energy method of Berselli and Gubinelli to the rough path setting and obtain  a global existence result for the equation~\eqref{eqn:RandomFilam-1}  when the  initial data is a geometric rough path, thus completing the analysis of \cite{[Bess+Gub_2005]}.

Recently there have been some deep progresses in the study of evolution equations in the space of controlled rough paths. In particular Hairer~\cite{Hairer} showed how to use controlled path theory to have well-posedness of a multidimensional Burgers type equation driven by additive space-time white noise and later~\cite{KPZ} used similar ideas to tackle the longstanding open problem of the Kardar--Parisi--Zhang equation for which he described a notion of solution for which existence and uniqueness can be proven.. The key technical tool to obtain these results has been the observation that the more singular non-linear term involved in the fixed-point argument has the same structure of the Biot-Savart relation~\eqref{eqn:RandomFilam-3} and thus can be similarly handled via controlled paths techniques.

A major open problem which remains largely unexplored is the rigorous study of the unregularized filament equation~\eqref{eqn:RandomFilamClassic}, whose singularity seems to defy any reasonable analytic approach. On this respect the very suggestive heuristic computations of Gallavotti~\cite{Gallavotti} seems to hint to the fact that rapid oscillations of the curve $\gamma$ could provide a natural regularization. Another very interesting problem is the study of the rough filament equation as an Hamiltonian system on the space of parametrization invariant paths on the lines of~\cite{arnold_topological_1998} and~\cite{[Brylinski-1992]}.

\bigskip
\textbf{Acknowledgments.} The authors would like to thank the anonymous referees for their careful reading and the detailed comments which helped us to improve the presentation.
\vfill\newpage

\section{Definition and Properties of Rough Path integrals}
In this section we will present the controlled path framework and we will  define the rough path integral and state some of its properties. We will mainly follow the papers \cite{[Gubinelli-2004]} and
\cite{[Bess+Gub_2005]}.

 In what follows by $V,W$ we will usually denote  Banach spaces, by $L(V,W)$ the Banach space of bounded linear maps from $V$ to $W$ and by $C(X,Y)$ the space of continuous function from $X$ to $Y$. Moreover we let $\mathbb{T}=[0,1]$ and
\begin{eqnarray*}
C_n(V)&=&\{f\in C(\mathbb{T}^n,V):\; f(t_1,t_2,\cdots,t_n)=0 \text{ if  $t_i=t_{i+1}$ for some $i=1,..,n-1$}\}.
\end{eqnarray*}
We will understand that $C(V) = C_1(V)$ and  we will consider the H\"older norm
$$
\vert f\vert_{\mu}=\supl_{a\in\mathbb{T}}|f(a)|+ \supl_{0\le a<b\le 1}\frac{\vert f(a)-f(b)\vert_{V}}{\vert a-b\vert^{\mu}}, \;\; f\in C_1(V);
$$
on $C_1(V)$ and the following semi-norms on the spaces 
 $C_2(V)$ and $ C_3(V)$
$$
\vert f\vert_{\mu}=\supl_{0\le a<b\le 1}\frac{\vert f(a,b)\vert_{V}}{\vert a-b\vert^{\mu}};\;\; f\in C_2(V), 
\qquad
\vert g\vert_{\mu}=\supl_{0\le a<b<c\le 1}\frac{\vert g(a,b,c)\vert_{V}}{\vert a-c\vert^{\mu}}, \;\; g\in C_3(V).
$$
The  associated normed spaces are defined by 
$$
C^\mu(V)= C_1^\mu(V) = \{f\in C_1(V):\, \vert f\vert_{\mu}<\infty\},
$$
$$
C_2^\mu(V) = \{f\in C_2(V):\, \vert f\vert_{\mu}<\infty\},
\qquad
C_3^\mu(V) = \{g\in C_3(V):\, \vert g\vert_{\mu}<\infty\}.
$$
For $n=1,2$ define the operators $\delta-\delta_n:C_n(V)\to C_{n+1}(V)$ as
$$
\delta_1 f(a,b) = f(b)-f(a),\qquad
\delta_2 g(a,b,c) = g(a,c)-g(a,b)-g(b,c)
$$
for all $a,b,c\in\mathbb{T}$. They satisfy the following fundamental property
$
\delta\delta f=0,
$ for all $f\in C_1(V)$. Let
$\mathcal{Z}C_2^\mu(V) = C_3^\mu(V)\cap\ker\delta $ and $\mathcal{B}C_2^\mu(V) = C_3^\mu(V)\cap\im\delta $. If $g\in \mathcal{Z} C_2^\mu(V)$ then  $g\in\mathcal{B}C_2^\mu(V)$, i.e. is there exists $f\in C_1(V)$ such that $\delta f = g$.
Then following  result  has been proved in
\cite{[Gubinelli-2004]}:
\begin{proposition}[Sewing lemma]
For any $\mu>1$ there exists an unique  linear map (Sewing map)
$\Lambda:\mathcal{Z}C_3^{\mu}(V)\to C_2^{\mu}(V)$ such that
$$
\delta\Lambda=\id_{\mathcal{Z}C_3^{\mu}(V)}.
$$
This map is continuous and we have
$$
\Vert \Lambda h\Vert_{\mu}\leq
\frac{1}{2^{\mu}-2}\Vert h\Vert_{\mu},\qquad h\in\mathcal{Z}C_3^{\mu}(V).
$$
\end{proposition}
Now, we define a class of paths for which rough path integral will
be defined.
\begin{definition}\label{def-controlled} Assume that $\nu \in (0,1)$ and  fix $X\in C^{\nu}(V)$. 
We say that a path $Y\in C(V)$ is weakly controlled by $X$ if
there exist a function $Y'\in C^{\nu}(L(V,V))$ such that the function 
\begin{equation}
R(\xi,\eta) = Y(\xi)-Y(\eta) - Y'(\eta)(X(\xi)-X(\eta)) ,\qquad\xi,\eta\in
\mathbb{T}, \label{eqn:WeakControlDef-1}
\end{equation}
belongs to $C_2^{2\nu}(V)$. We denote $\mathcal{D}_X$ the set of all the pairs $(Y,Y')$ satisfying~\eqref{eqn:WeakControlDef-1}. This is a vector space which can be endowed with the semi-norm $\Vert \cdot\Vert_{\mathcal{D}_X}$  given by the following formula
\begin{equation}
\Vert (Y,Y')\Vert_{\mathcal{D}_X}=\vert Y'\vert_{C^{\nu}}+\Vert R\Vert_{C_2^{2\nu}},\label{eqn:D_XNorm-1}
\end{equation}
Furthermore, let us define a  norm $\Vert \cdot\Vert_{\mathcal{D}_X}^\ast$  in
$\mathcal{D}_X$ by
\begin{equation}
\Vert (Y,Y')\Vert_{\mathcal{D}_X}^\ast=\Vert Y\Vert_{\mathcal{D}_X}+\sup_{\xi\in \mathbb{T}} |Y(\xi)|. \label{eqn:D_XNorm-2}
\end{equation}
With these definitions
$(\mathcal{D}_X,\Vert \cdot\Vert_{\mathcal{D}_X}^\ast)$ is a Banach space. While not necessary for the integration theory it will convenient for our intended application to assume that all our controlled paths take values in  $V=\mathbb{R}^3$ and are closed, that is $Y(0)=Y(1)$.
\end{definition}

In the present paper an element $(Y,Y^\prime)\in \mathcal{D}_X$ can always be canonically identified from its first component $Y$ by the mean of standard procedures discussed below. In this context we denote by $R^Y$ the corresponding remainder: $R^Y = \delta Y - Y' \delta X$. Moreover we will often
omit to specify $Y^\prime$ when it is clear from the context and write
$\Vert Y\Vert_{\mathcal{D}_X}$ instead of $\Vert (Y,Y^\prime)\Vert_{\mathcal{D}_X}$.

We will repeatedly use the basic estimate
\begin{equation}
\Vert Y\Vert_{C^{\nu}}\leq
\Vert Y\Vert_{\mathcal{D}_X}^*(1+\Vert X\Vert_{C^{\nu}}).\label{eqn:HolderIneq}
\end{equation}

Controlled paths are stable by mapping via sufficiently regular functions. 
\begin{lemma}\label{lem:RegularMap}
Let $\phi\in C^2(\Rnu^3,\Rnu^3)$ and $(Y,Y')\in \mathcal{D}_X$.
Then
\begin{equation}
(W,W^\prime):=(\phi(Y),\phi^\prime(Y)Y')\in
\mathcal{D}_X\label{eqn:RegularMap-1}
\end{equation}
and the corresponding remainder has the following representation
\begin{eqnarray}
R^{W}(\xi,\eta)&=&\phi^\prime(Y(\xi))R^Y(\xi,\eta)+
(Y(\eta)-Y(\xi))\times
\nonumber
\\&& \times\intl_0^1[\nabla\phi(Y(\xi)+r(Y(\eta)-Y(\xi)))-\nabla\phi(Y(\xi))]dr,\qquad\xi,\eta\in
\mathbb{T}.
\label{eqn:RegularMap-2}
\end{eqnarray}
Furthermore, there exists a
constant $K\geq 1$ such that
\begin{equation}
\Vert \phi(Y)\Vert_{\mathcal{D}_X}\leq
K\Vert \nabla\phi\Vert_{C^1}\Vert Y\Vert_{\mathcal{D}_X}(1+\Vert Y\Vert_{\mathcal{D}_X})(1+\Vert X\Vert_{\nu})^2.\label{eqn:RegularMap-3}
\end{equation}
Moreover, if $(\tilde{Y},\tilde{Y}')\in \mathcal{D}_{\tilde{X}}$
and
$
(\tilde{W},\tilde{W}^\prime):=(\phi(\tilde{Y}),\phi^\prime(\tilde{Y})\tilde{Y}')
$
then
\begin{equation}
\begin{split}
\label{eqn:RegularMap-4}
\Vert W^\prime-\tilde{W}^\prime\Vert_{\nu}+\Vert R^W-R^{\tilde{W}}\Vert_{2\nu}+\Vert W-\tilde{W}\Vert_{\nu}\leq\\
C(\Vert X-\tilde{X}\Vert_{\nu}+\Vert Y^\prime-\tilde{Y}^\prime\Vert_{\nu}+\Vert R^Y-R^{\tilde{Y}}\Vert_{2\nu}+\Vert Y-\tilde{Y}\Vert_{\nu})
\end{split}
\end{equation}
with
\begin{equation}
C=K\Vert \phi\Vert_{C^3}(1+\Vert X\Vert_{C^{\nu}}+\Vert \tilde{X}\Vert_{C^{\nu}})^3\vert (1+\Vert Y\Vert_{\mathcal{D}_X}+\Vert \tilde{Y}\Vert_{\mathcal{D}_{\tilde{X}}})^2.\label{eqn:RegularMap-4'}
\end{equation}
When $X=\tilde{X}$ we have
\begin{equation}
\label{eqn:RegularMap-4''}
\begin{split}
\Vert \phi(Y)-\phi(\tilde{Y})\Vert_{\mathcal{D}_X}\leq
& K\Vert \nabla\phi\Vert_{C^2}\Vert Y\Vert_{\mathcal{D}_X}\\
(1+&\Vert Y\Vert_{\mathcal{D}_X}+\Vert \tilde{Y}\Vert_{\mathcal{D}_X})^2(1+\Vert X\Vert_{C^{\nu}})^4
\Vert Y-\tilde{Y}\Vert_{\mathcal{D}_X}.
\end{split}
\end{equation}
\end{lemma}
\begin{proof}
See \cite{[Gubinelli-2004]}, Proposition $4$ for all statements of
the Lemma, except \eqref{eqn:RegularMap-2} (which is actually also
proven, though not stated explicitly). Let us show
\eqref{eqn:RegularMap-2}. Denote
$y(r)=Y(\xi)+r(Y(\eta)-Y(\xi)),r\in[0,1]$. Then
\begin{eqnarray}
\label{eqn:RegularMap-2proof}
&&
\phi(y(1))-\phi(y(0))=\intl_0^1\phi^\prime(y(r))y^\prime(r)dr
\\
\nonumber
&=&\suml_k(Y^k(\eta)-Y^k(\xi))\intl_0^1\frac{\partial\phi}{\partial
x_k}(y(r))dr
=\suml_k\frac{\partial\phi}{\partial
x_k}(Y(\xi))(Y^k(\eta)-Y^k(\xi))\nonumber
\\
&+&\suml_k(Y^k(\eta)-Y^k(\xi))\intl_0^1[\frac{\partial\phi}{\partial
x_k}(y(r))-\frac{\partial\phi}{\partial x_k}(Y(\xi))]dr\nonumber
\\
&=&\suml_{k,l}\frac{\partial\phi}{\partial
x_k}(Y(\xi))(Y^\prime)^{kl}(X^l(\eta)-X^l(\xi))
+\suml_{k}\frac{\partial\phi}{\partial
x_k}(Y(\xi))(R^Y)^k(\xi,\eta)\nonumber
\\
&+&\suml_k(Y^k(\eta)-Y^k(\xi))\intl_0^1[\frac{\partial\phi}{\partial
x_k}(y(r))-\frac{\partial\phi}{\partial
x_k}(Y(\xi))]dr,
\nonumber
\end{eqnarray}
and the result follows.
\end{proof}
In order to define the integral for paths controlled by $X$  we need a further ingredient: the existence of basic integrals of $X$ with respect to $X$ itself. These objects are encoded into what is refereed to as a \emph{rough path} as follows.
\begin{definition}\label{def:RoughPath}
Assume that  $\nu>{1}/{3}$. We say that couple
$\mathbb{X}=(X,\mathbb{X}^2)$ where $X\in C^{\nu}(\mathbb{T},\Rnu^3)$ and
$\mathbb{X}^2\in C_2^{2\nu}(L(\Rnu^3,\Rnu^3))$ is a $\nu$-rough
path if the following condition is satisfied:
\begin{equation}\label{eqn:CompatCondition}
\mathbb{X}^{2}(\xi,\rho)-\mathbb{X}^{2}(\xi,\eta)-\mathbb{X}^{2}(\eta,\rho)
=(X(\xi)-X(\eta))\otimes(X(\eta)-X(\rho)), \quad \xi,\eta,\rho\in \mathbb{T}.
\end{equation}
\end{definition}
\begin{remark}
If $\nu>1/2$ the solution $\mathbb{X}^{2}$ to~\eqref{eqn:CompatCondition} is unique. Indeed, assume that there exists another
$\mathbb{X}_1^{2}$ which satisfies definition \ref{def:RoughPath}.
Put $G(\xi)=\mathbb{X}^{2}(\xi,0)-\mathbb{X}_1^{2}(\xi,0)$. Then
by condition \eqref{eqn:CompatCondition}
$$
\mathbb{X}^{2}(\xi,\rho)-\mathbb{X}_1^{2}(\xi,\rho)=G(\xi)-G(\rho),
$$
and, since $\mathbb{X}^{2}\in C_2^{2\nu}$, $G$ is a H\"{o}lder
function of order bigger  than $1$. Hence,  $G=0$ and thus 
$\mathbb{X}_1^{2}=\mathbb{X}^{2}$.

If $\nu>1$ and $\mathbb{X}$ is a $\nu$-rough path, then
$\mathbb{X}$ is given by the pair of constants $(X(0),0)$. Indeed,
in this case $X$ is H\"older function with exponent more than $1$
i.e. $X$ is equal to a constant function  $X(0)$ and the only possible choice for the second component is $\mathbb{X}^{2}=0$.

If $\nu\in({1}/{2},1]$ then $\mathbb{X}^2$ is
uniquely determined by $X$. Indeed, if we let
\begin{equation}
\mathbb{X}^{2,ij}(\xi,\eta)=\intl_{\xi}^{\eta}(X^i_{\rho}-X^i_{\eta})dX^j_{\rho},\xi,\eta\in
\mathbb{T}\; i,j=1,2,3, \qquad\label{eqn:YoungIntegralCase}
\end{equation}
where the integral is understood in the sense of Young (see
\cite{[Young-1936]}) it is possible to show that $\mathbb{X}^{2}\in C^{2\nu}_2$  and that this definition satisfies conditions of
Definition~\ref{def:RoughPath}. Since $\nu > 1/2$ it is the only possible choice.
\end{remark}
\begin{definition}\label{ass:RoughPathApproximability}
We say that our $\nu$-rough path $(X,\mathbb{X}^2)$  is a
geometric $\nu$-rough path if there exist a sequence
$(X_n,\mathbb{X}_n^2)$ such that
$
X_n\in C(\mathbb{T},\Rnu^3)
$ is piecewise smooth, and 
\begin{equation}
\lim_{n\to\infty}\big[\Vert X_n-X\Vert_{\nu}+\Vert \mathbb{X}_n^2-\mathbb{X}^2\Vert_{2\nu}\big]=
0,\label{eqn:RoughPathTopology}
\end{equation}
where
$$
\mathbb{X}_n^{2,ij}(\xi,\eta)=\intl_{\xi}^{\eta}(X_n^i({\rho})-X_n^i({\eta}))dX_n^j({\rho}),\qquad \xi,\eta\in
\mathbb{T},i,j=1,2,3.
$$
\end{definition}
\begin{example}
Let $\{B_t\}_{t\in [0,1]}$ be the standard $3$-dimensional Brownian
bridge such that $B_0=B_1=x_0$ and let $\mathbb{B}^{2,ij}$, $i,j=1,2,3$,  be the
area processes defined by
$$
\mathbb{B}^{2,ij}(\xi,\eta)=\intl_{\xi}^{\eta}(B^i_{\rho}-B^i_{\eta})dB^j_{\rho},
$$
where the integral can be understood either in the Stratonovich or in the It\^o
sense. Then, the couple $(B,\mathbb{B}^2)$ is a $\nu$-rough path for all $\nu\in(1/3,1/2)$ (see \cite{[Bess+Gub_2005]}*{p.1849}). Moreover, if the integral is understood in the Stratonovich sense,  then $(B,\mathbb{B}^2)$ is also a
geometric $\nu$-rough path.  Indeed, it follows from Theorem
3.1 in \cite{[Friz-2005]} that one can approximate $X$ with
piecewise linear dyadic $(X_n)_{n\ge 1}$ in the sense of Definition~\ref{ass:RoughPathApproximability} where the limit is understood in the almost sure sense.
\end{example}
From now on we suppose that the geometric $\nu$-rough path
$\mathbb{X}=(X,\mathbb{X}^2)$ with $\nu>1/3$ and the corresponding Banach space
$\mathcal{D}_X$ are fixed.  For a  finite
partition $\pi=\{\xi_0=\xi<\xi_1<\cdots<\xi_n=\eta\}$ be   of the interval  $[\xi,\eta]$, let  $d(\pi)=\supl_i\vert \xi_{i+1}-\xi_i\vert $ denote the mesh size of the partition $\pi$.
\begin{lemma}\label{lem:IntegralDef}
  If $Y,Z\in \mathcal{D}_X$ then the limit
\begin{equation}
\lim_{d(\pi)\to
0}\sum_{i=0}^{n-1}[Y(\xi_i)(Z(\xi_{i+1})-Z(\xi_i))+Y^\prime(\xi_i)Z^\prime(\xi_i)\mathbb{X}^2(\xi_{i+1},\xi_i)] =: \int_{\xi}^{\eta}YdZ\label{eqn:IntegralDef-1}
\end{equation}
exists and defined the integral on the r.h.s.
\end{lemma}
\begin{proof}
This statement is a direct consequence of the existence and properties of the sewing map and of the definition of the controlled path space, see \cite{[Gubinelli-2004]}, Theorem $1$. For the sake of the presentation we sketch a proof. Let us denote
$$
A(\xi',\eta')= Y(\xi')(Z(\eta')-Z(\xi'))+Y^\prime(\xi')Z^\prime(\xi')\mathbb{X}^2(\eta',\xi'), \;\; \eta',\xi'\in\mathbf{T}.
$$
Let us note that, with the assumptions of the Lemma,  $\delta A \in C^{3\nu}_3$. Hence,  provided that $3\nu> 1$, we can apply the sewing map and we have the decomposition
$
A = \delta I +  \Lambda \delta A
$
where $I\in C_1$. But this implies that
$$
\lim_{d(\pi)\to
0}\sum_{i=0}^{n-1} A(\xi_i,\xi_{i+1}) =\lim_{d(\pi)\to
0}\sum_{i=0}^{n-1} (I(\xi_{i+1})-I(\xi_i)) +  \lim_{d(\pi)\to
0}\sum_{i=0}^{n-1} (\Lambda \delta A(\xi_i,\xi_{i+1})).
$$
The first term on the r.h.s. telescopes and since $|\Lambda \delta A(\xi_i,\xi_{i+1})| \lesssim |\xi_{i+1}-\xi_i|^{3\nu}$ the second term on the r.h.s converges to zero proving the existence of the limit and moreover 
$$
\lim_{d(\pi)\to
0}\sum_{i=0}^{n-1} A(\xi_i,\xi_{i+1}) = I(\eta)-I(\xi).
$$
\end{proof}
The explicit representation in terms of the sewing map can be used to prove the various estimates for the rough integral below.
\begin{remark}\label{rem:SmoothCase}
In the case of $\nu>1/2$ the line integral defined in 
Lemma \ref{lem:IntegralDef} is reduced to the Young definition of the
line integral $\int YdZ$. Indeed, it is enough to notice that
second term in formula \eqref{eqn:IntegralDef-1} is of the order
$O(\vert \xi_{i+1}-\xi_i\vert^{2\nu})$, $2\nu>1$. In this case the line integral
does not depend upon $Y^\prime$, $Z^\prime$.
\end{remark}

Integrals of controlled paths have very nice continuity property with respect to variations of the controlled path and also with respect variations of the  rough path on which the controlled path space is modelled.

In the following $\mathbb{X}=(X,\mathbb{X}^2)$ and $\tilde{\mathbb{X}}=(\tilde{X},\tilde{\mathbb{X}}^2)$ are two $\nu$-rough paths and  $Y\in\mathcal{D}_X$ and $\tilde{Y}\in\mathcal{D}_{\tilde{X}}$ two path respectively controlled by $X$ and $\tilde X$. In this case we introduce the following distance
\begin{equation}\label{eqn-distance}
D(Y,\tilde{Y}) = \Vert X-\tilde{X}\Vert_{C^{\nu}}+\Vert \mathbb{X}^2-\mathbb{\tilde{X}}^2\Vert_{C_2^{2\nu}}+\Vert Y^\prime-\tilde{Y}^\prime\Vert_{C^{\nu}}+\Vert R^Y-R^{\tilde{Y}}\Vert_{C_2^{2\nu}}+\Vert Y-\tilde{Y}\Vert_{C^{\nu}}.
\end{equation}

\begin{lemma}\label{lem:IntegralProp}
Assume $Y,W\in \mathcal{D}_X$, $\tilde{Y},\tilde{W}\in
\mathcal{D}_{\tilde{X}}$. Define maps $Q,\tilde{Q}:\mathbb{T}^2\to\Rnu$ by
the following identities
\begin{equation}
Q(\eta,\xi):=\int_{\xi}^{\eta}YdW-Y(\xi)(W(\eta)-W(\xi))-Y^\prime(\xi)W^\prime(\xi)\mathbb{X}^2(\eta,\xi),\qquad\eta,\xi\in
\mathbb{T},\label{eqn:IntegralDef-2}
\end{equation}
and a similar expression for $\tilde{Q}$ with $Y,Y',W,W',\mathbb{X}^2$ replaced by $\tilde{Y},\tilde{Y}',\tilde{W},\tilde{W}',\tilde{\mathbb{X}}^2$.
Then $Q,\tilde{Q}\in C_2^{3\nu}$ and there exists a constant
$C=C(\nu)>0$ such that for all $Y,W\in \mathcal{D}_X$
\begin{equation}
\Vert Q\Vert_{C_2^{3\nu}}\leq
C(1+\Vert X\Vert_{C^{\nu}}+\Vert \mathbb{X}^2\Vert_{C_2^{2\nu}})\Vert Y\Vert_{\mathcal{D}_X}\Vert W\Vert_{\mathcal{D}_X}.\label{eqn:IntegralDef-3}
\end{equation}
Furthermore,
\begin{multline}\label{eqn:IntegralDef-5}
\Vert Q-\tilde{Q}\Vert_{C_2^{3\nu}}\leq
C(1+\Vert X\Vert_{C^{\nu}}+\Vert \mathbb{X}^2\Vert_{C_2^{2\nu}})\\
((\Vert Y\Vert_{\mathcal{D}_X}+\Vert \tilde{Y}\Vert_{\mathcal{D}_{\tilde{X}}})\eps_W+
(\Vert W\Vert_{\mathcal{D}_X}+\Vert \tilde{W}\Vert_{\mathcal{D}_{\tilde{X}}})\eps_Y+\eps_X).
\end{multline}
where
\[
\eps_Y=\Vert Y^\prime-\tilde{Y}^\prime\Vert_{C^{\nu}}+\Vert R^Y-R^{\tilde{Y}}\Vert_{C_2^{2\nu}}+\Vert Y-\tilde{Y}\Vert_{C^{\nu}},
\]
\[
\eps_W=\Vert W^\prime-\tilde{W}^\prime\Vert_{C^{\nu}}+\Vert R^W-R^{\tilde{W}}\Vert_{C_2^{2\nu}}+\Vert W-\tilde{W}\Vert_{C^{\nu}},
\]
\[
\eps_X=(\Vert Y\Vert_{\mathcal{D}_X}+\Vert \tilde{Y}\Vert_{\mathcal{D}_{\tilde{X}}})(\Vert W\Vert_{\mathcal{D}_X}+\Vert \tilde{W}\Vert_{\mathcal{D}_{\tilde{X}}})
(\Vert X-\tilde{X}\Vert_{C^{\nu}}+\Vert \mathbb{X}^2-\mathbb{\tilde{X}}^2\Vert_{C_2^{2\nu}}).
\]
\end{lemma}
\begin{proof}
See \cite{[Gubinelli-2004]}, Theorem 1. For formula
\eqref{eqn:IntegralDef-5} see   \cite{[Gubinelli-2004]}, p.104,
formula (27).
\end{proof}
By Lemmata \ref{lem:IntegralDef} and \ref{lem:RegularMap} for any
$A\in C^2(\Rnu^3,L(\Rnu^3,\Rnu^3)),Y\in\mathcal{D}_X$ we can a
define a map  $V^Y:\Rnu^3\to\Rnu$  by the means of the rough path integral as follows
\begin{equation}
V^Y(x):=\int_{0}^1A(x-Y(\xi))d Y(\xi),\qquad x\in\Rnu^3.
\end{equation}
For maps $F:\mathbb{R}^3\to W$ and $n\ge 0$ we denote by $\|F\|_{C^n}$ the usual norms
$$
\|F\|_{C^n} =   \sup_{0\le  |\alpha|\le n} \sup_{x\in\mathbb{R}^3}   |\nabla^\alpha F(x)|.
$$

 We have following bounds on the regularity of $V^Y$.
\begin{lemma}\label{lem:ParametricRoughPath}
 There exists $C_1=C_1(\nu),C_2=C_2(\mathbb{X})$ such that for
any integer $n\geq 0$ and all $Y\in \mathcal{D}_X$, $\tilde{Y}\in \mathcal{D}_{\tilde{X}}$,
\begin{equation}
\Vert \nabla^nV^Y\Vert_{L^{\infty}}\leq
4C_1C_2^3\Vert \nabla^{n+1}A\Vert_{C^1}\Vert Y\Vert_{\mathcal{D}_X}^2(1+\Vert Y\Vert_{\mathcal{D}_X})\label{eqn:ParametricRoughPath-1}
\end{equation}
and
\begin{multline}\label{eqn:ParametricRoughPath-2general}
\Vert \nabla^nV^Y-\nabla^nV^{\tilde{Y}}\Vert_{L^{\infty}}\leq
C(\nu)\Vert A\Vert_{C^{n+3}}C_X^4(1+\Vert Y\Vert_{\mathcal{D}_X}+\Vert \tilde{Y}\Vert_{\mathcal{D}_{\tilde{X}}})^3 D(Y,\tilde{Y})
\end{multline}
where
$$
C_X=1+\Vert X\Vert_{C^{\nu}}+\Vert \tilde{X}\Vert_{C^{\nu}}+\Vert \mathbb{X}^2\Vert_{C_2^{2\nu}}+\Vert \mathbb{\tilde{X}}^2\Vert_{C_2^{2\nu}}.
$$

In the case of $X=\tilde{X}$, inequality \eqref{eqn:ParametricRoughPath-2general} can be rewritten as
\begin{equation}
\Vert \nabla^nV^Y-\nabla^nV^{\tilde{Y}}\Vert_{L^{\infty}}\leq
16C_1C_2^3\Vert \nabla^{n+1}A\Vert_{C^2}\Vert Y\Vert_{\mathcal{D}_X}(1+\Vert Y\Vert_{\mathcal{D}_X})^2\Vert Y-\tilde{Y}\Vert_{\mathcal{D}_X}^*.
\label{eqn:ParametricRoughPath-2}
\end{equation}
\end{lemma}
\begin{proof}
By the results~\cite[Lemma $7$]{[Bess+Gub_2005]} it is known that $V^Y \in C^{n}$ if $A\in C^{n+2}$ and alsto that eqns.~\eqref{eqn:ParametricRoughPath-1} and
\eqref{eqn:ParametricRoughPath-2} hold. Now we will show
\eqref{eqn:ParametricRoughPath-2general}. It is enough to consider
the case of $n=0$. By eq.~\eqref{eqn:IntegralDef-3}  we have
\begin{eqnarray}
V^Y(x)-V^{\tilde{Y}}(x)&=&A(x-Y(0))(Y(1)-Y(0))
-A(x-\tilde{Y}(0))(\tilde{Y}(1)-\tilde{Y}(0))\nonumber\\
&&\quad + (A(x-Y))^\prime(0)Y^\prime(0)\mathbb{X}^2(0,1)
-(A(x-\tilde{Y}))^\prime(0)\tilde{Y}^\prime(0)\mathbb{\tilde{X}}^2(0,1)
\nonumber\\
&&\quad + Q^x(0,1)-\tilde{Q}^x(0,1),
\nonumber
\end{eqnarray}
where $Q^x$ and $\tilde{Q}^x$ (given by the eq.~\eqref{eqn:IntegralDef-3}) satisfy
inequality \eqref{eqn:IntegralDef-5}. Recall that $Y(1)=Y(0)$ and $\tilde{Y}(1)=\tilde{Y}(0)$, hence we have
\begin{equation}\label{eqn:PRaux-1}
\begin{split}
\vert V^Y-V^{\tilde{Y}}\vert_{L^{\infty}}\leq &
\supl_x\vert (A(x-Y))^\prime(0)Y^\prime(0)\mathbb{X}^2(0,1)-
(A(x-\tilde{Y}))^\prime(0)\tilde{Y}^\prime(0)\mathbb{\tilde{X}}^2(0,1)\vert \\
& +\supl_x\vert Q^x(0,1)-\tilde{Q}^x(0,1)\vert .
\end{split}
\end{equation}
For the first term on the r.h.s. we have
\begin{equation}
\label{eqn:PRaux-2}
\begin{split}
\vert (&A(x-Y))^\prime(0)Y^\prime(0)\mathbb{X}^2(0,1)-
(A(x-\tilde{Y}))^\prime(0)\tilde{Y}^\prime(0)\mathbb{\tilde{X}}^2(0,1)\vert \\
&\leq
\vert (\nabla A(x-Y(0))Y^\prime(0)Y^\prime(0)-\nabla A(x-\tilde{Y}(0))\tilde{Y}^\prime(0)\tilde{Y}^\prime(0))\mathbb{X}^2(0,1)\vert \\
&+\vert \nabla A(x-\tilde{Y}(0))\tilde{Y}^\prime(0)\tilde{Y}^\prime(0)(\mathbb{X}^2(0,1)-\mathbb{\tilde{X}}^2(0,1))\vert \\
&\leq \Vert \mathbb{X}^2\Vert_{C_2^{2\nu}}\vert \nabla A(x-Y(0))Y^\prime(0)Y^\prime(0)-\nabla A(x-\tilde{Y}(0))\tilde{Y}^\prime(0)\tilde{Y}^\prime(0)\vert \\
&+ \Vert A\Vert_{C^2}\Vert Y^\prime\Vert_{L^{\infty}}^2\Vert \mathbb{X}^2-\mathbb{\tilde{X}}^2\Vert_{C_2^{2\nu}}\\
&\leq
\Vert \mathbb{X}^2\Vert_{C_2^{2\nu}}\Vert Y^\prime\Vert_{L^{\infty}}^2\Vert A\Vert_{C^2}\Vert Y^\prime-\tilde{Y}^\prime\Vert_{L^{\infty}}
+\Vert \mathbb{X}^2\Vert_{C_2^{2\nu}}\Vert A\Vert_{C^1}(\Vert Y^\prime\Vert_{L^{\infty}}+\Vert \tilde{Y}^\prime\Vert_{L^{\infty}})\Vert Y^\prime-\tilde{Y}^\prime\Vert_{L^{\infty}}\\
&+\Vert A\Vert_{C^2}\Vert Y^\prime\Vert_{L^{\infty}}^2\Vert \mathbb{X}^2-\mathbb{\tilde{X}}^2\Vert_{C_2^{2\nu}}.
\end{split}
\end{equation}
By~\eqref{eqn:IntegralDef-5} we can estimate second term as
follows
\begin{equation}
\label{eqn:PRaux-3}
\begin{aligned}
\vert Q^x-\tilde{Q}^x\vert_{C_2^{3\nu}}&\leq
C\Big[(\Vert A(x-Y)\Vert_{\mathcal{D}_X}+\Vert A(x-\tilde{Y})\Vert_{\mathcal{D}_{\tilde{X}}})\eps_Y\nonumber\\&+
(\Vert Y\Vert_{\mathcal{D}_X}+\Vert \tilde{Y}\Vert_{\mathcal{D}_{\tilde{X}}})\eps_A+\eps_X\Big],
\end{aligned}
\end{equation}
where
\begin{align*}
\eps_Y&=\Vert Y-\tilde{Y}\Vert_{C^{\nu}}+\Vert Y^\prime-\tilde{Y}^\prime\Vert_{C^{\nu}}+\Vert R^Y-R^{\tilde{Y}}\Vert_{C_2^{2\nu}},\\
\eps_A&=\Vert A(x-Y)-A(x-\tilde{Y})\Vert_{C^{\nu}}+\Vert A(x-Y)^\prime-A(x-\tilde{Y})^\prime\Vert_{C^{\nu}}\\
&+\Vert R^{A(x-Y)}-R^{A(x-\tilde{Y})}\Vert_{C_2^{2\nu}},\\
\eps_X&=(\Vert A(x-Y)\Vert_{\mathcal{D}_X}+\Vert A(x-\tilde{Y})\Vert_{\mathcal{D}_{\tilde{X}}})\\
&\times(\Vert Y\Vert_{\mathcal{D}_X}+\Vert \tilde{Y}\Vert_{\mathcal{D}_{\tilde{X}}})
(\vert X-\tilde{X}\vert_{C^{\nu}}+\vert \mathbb{X}^2-\mathbb{\tilde{X}}^2\vert_{C_2^{2\nu}}).
\end{align*}
By formula \eqref{eqn:RegularMap-4} we can estimate $\eps_A$ as
follows
\begin{align}
 \eps_A &\leq
K\Vert A\Vert_{C^3}(1+\Vert X\Vert_{C^{\nu}}+\Vert \tilde{X}\Vert_{C^{\nu}})^3(1+\Vert Y\Vert_{D_X}+\Vert \tilde{Y}\Vert_{D_{\tilde{X}}})^2\nonumber\\
&\times(\Vert X-\tilde{X}\Vert_{C^{\nu}}+\Vert Y-\tilde{Y}\Vert_{C^{\nu}}+\Vert Y^\prime-\tilde{Y}^\prime\Vert_{C^{\nu}}+\Vert R^Y-R^{\tilde{Y}}\Vert_{C_2^{2\nu}}).\label{eqn:PRaux-4}
\end{align}
By inequality \eqref{eqn:RegularMap-3} we infer that
\begin{eqnarray}
\Vert A(x-Y)\Vert_{\mathcal{D}_X}\leq
K\Vert A\Vert_{C^2}\Vert Y\Vert_{D_X}(1+\Vert Y\Vert_{D_X})(1+\Vert X\Vert_{C^{\nu}})^2,\label{eqn:PRaux-5}
\end{eqnarray}
and similarly,
\begin{eqnarray}
\Vert A(x-\tilde{Y})\Vert_{\mathcal{D}_{\tilde{X}}}\leq
K\Vert A\Vert_{C^2}\Vert \tilde{Y}\Vert_{D_{\tilde{X}}}(1+\Vert \tilde{Y}\Vert_{D_{\tilde{X}}})(1+\Vert \tilde{X}\Vert_{C^{\nu}})^2.\label{eqn:PRaux-6}
\end{eqnarray}
Therefore, combining inequalities \eqref{eqn:PRaux-3} with \eqref{eqn:PRaux-4},
\eqref{eqn:PRaux-5} and \eqref{eqn:PRaux-6} we get
\begin{multline}\label{eqn:PRaux-7}
\Vert Q^x-\tilde{Q}^x\Vert_{C_2^{3\nu}}\leq C(\nu)\Vert A\Vert_{C^{n+3}}(1+\Vert X\Vert_{C^{\nu}}+\Vert \tilde{X}\Vert_{C^{\nu}})^4(1+\Vert Y\Vert_{\mathcal{D}_X}+\Vert \tilde{Y}\Vert_{\mathcal{D}_{\tilde{X}}})^3\\
(\Vert X-\tilde{X}\Vert_{C^{\nu}}+\Vert \mathbb{X}^2-\mathbb{\tilde{X}}^2\Vert_{C_2^{2\nu}}+\Vert Y^\prime-\tilde{Y}^\prime\Vert_{C^{\nu}}+\Vert R^Y-R^{\tilde{Y}}\Vert_{C_2^{2\nu}}+\Vert Y-\tilde{Y}\Vert_{C^{\nu}}).
\end{multline}
Hence, the result follows from \eqref{eqn:PRaux-2} and
\eqref{eqn:PRaux-7}.
\end{proof}

\section{The Evolution Problem for a Rough Filament}
Let $\mathcal{D}_{\mathbb{X},T}=C([0,T],\mathcal{D}_X)$, where the space $\mathcal{D}_X$ has been defined in Definition \ref{def-controlled},  be a
vector space with the usual supremum norm
\begin{equation}
\Vert F\Vert_{\mathcal{D}_{\mathbb{X},T}}=\supl_{t\in[0,T]}\Vert F(t)\Vert_{\mathcal{D}_X}^*\label{eqn:NormDef}.
\end{equation}
Obviously $\mathcal{D}_{\mathbb{X},T}$ is a Banach space. For an element $\gamma \in \mathcal{D}_{\mathbb{X},T}$, the time-dependent vector field
$$
(t,x)\mapsto u^{\gamma(t)}(x) =\int_0^1 \nabla \phi(x-\gamma(t,\xi)) \times d\gamma(t,\xi)
$$
is well defined by the rough integral and it is meaningful to consider the  Cauchy problem~\eqref{eqn:RandomFilam-1} in $\mathcal{D}_{\mathbb{X},T}$.
The following local existence and uniqueness result been proved in \cite{[Bess+Gub_2005]}, see Theorem
$3$,~p.1842.
\begin{theorem}\label{thm:LocalExistUniq}
Assume $\phi\in C^6(\Rnu^3,\Rnu)$, $\nu\in ({1}/{3},1)$,
$\mathbb{X}=(X,\mathbb{X}^2)$ is a $\nu$-rough path, $\gamma_0\in
\mathcal{D}_X$. Then there exists a time
$T_0=T_0(\nu,\Vert \phi\Vert_{C^5},\mathbb{X})>0$ such that the problem~\eqref{eqn:RandomFilam-1} has unique
solution in the space  $\mathcal{D}_{\gamma_0,T_0}\subset\mathcal{D}_{\mathbb{X},T_0}$. Moveover if $\mathbb{X}, \tilde{\mathbb{X}}$ are two $\nu$-rough paths, $\gamma_0\in \mathcal{D}_X$, $\tilde \gamma_0\in
\mathcal{D}_{\tilde{X}}$ and $\gamma,\tilde \gamma$  two solutions of~\eqref{eqn:RandomFilam-1} respectively in $\mathcal{D}_{\mathbb{X},T_0}$ and $\mathcal{D}_{\tilde{\mathbb{X}},\tilde T_0}$ then there exists a constant $C=C(\mathbb{X},\tilde{\mathbb{X}},\phi,\nu,\|\gamma_0\|_{\mathcal{D}_X}^*,\|\tilde \gamma_0\|_{\mathcal{D}_{\tilde{X}}}^*)$ such that
$$
\sup_{0\le t\le \min(T_0,\tilde T_0)} D(\gamma(t),\tilde \gamma(t)) \le C D(\gamma(0),\tilde \gamma(0)) .
$$
\end{theorem}
This result can be easily proven by applying the Banach  Fixed Point Theorem in the space $\mathcal{D}_{\mathbb{X},T}$  with a sufficiently small $T$. The requirements on $\phi$ can be easily understood as follows. In order for $u^\gamma$ to be well defined the function $\nabla \phi$ has to be twice differentiable. Moreover in order to apply the fixed point theorem the function $(t,\xi) \mapsto u^{\gamma(t)}(\gamma(t,\xi))$ has to belong to  $\mathcal{D}_{\mathbb{X},T}$ and this can enforced by requiring that the map $x \mapsto u^{\gamma(t)}(x)$ has to be $C^2$ in space, uniformly in time. This requires $\nabla \phi$ to be at least $4$ times differentiable.  Furthermore we need  a strict contraction of the fixed-point map which holds at the price of an additional derivative of $\nabla \phi$ ending up with the requirement that $\phi$ has to be six times differentiable. It is not obvious how to improve on this requirement for a general rough  initial data and $\nu\in(1/3,1/2)$.

The arguments used in Theorem~\ref{thm:LocalExistUniq} does not allow to have a global solution. The aim of the present paper is to enforce further structure condition on the kernel $\phi$ in order to be able to use the energy method introduced in~\cite{[Bers+Gub_2007]}.

To prove the  global existence result for the problem~\eqref{eqn:RandomFilam-1} we need the following additional hypothesis
\begin{hypothesis}\label{hyp:PhiAssumptions}
 Assume that the function $\phi$ appearing in eq.~\eqref{eqn:RandomFilam-3} has a Fourier transform $\hat\phi:\Rnu^3\to\Rnu$ which is real, even and
non-negative and satisfying the integrability condition
$$
\int_{\Rnu^3}(1+\vert k\vert^2)^2\hat{\phi}(k)dk<\infty.
$$
\end{hypothesis}
\begin{example}
The function $\phi_\mu$, $\mu>0$ defined by
$$\phi_\mu(\cdot)=\frac{1}{(\vert \cdot\vert^2+\mu^2)^{\half}}$$
is smooth and satisfies Hypothesis \ref{hyp:PhiAssumptions}, see p.6
of \cite{[Bers+Gub_2007]}. This function is also of $C^\infty$-class  so it satisfy also the assumptions of Theorem~\ref{thm:LocalExistUniq} and as a result is an explicit example which satisfy all the assumptions needed for the following global existence result.
\end{example}
\begin{theorem}\label{thm:GlobalExistUniq}
Assume that  $\gamma_0$ is a geometric $\nu$-rough path, $\phi\in
C^6(\Rnu^3,\Rnu)$ satisfies   Hypothesis
\ref{hyp:PhiAssumptions} and  $\nu\in({1}/{3},1)$. Then for every $T>0$, the problem~\eqref{eqn:RandomFilam-1} has unique
solution in $\mathcal{D}_{\gamma_0,T}$.
\end{theorem}
We will need the following definition.
\begin{definition}
Let $\phi\in C^4(\Rnu^3,\Rnu)$, $\gamma\in \mathcal{D}_X$. Let
\begin{equation*}
\psi^{\gamma}(x)=\int_0^1\phi(x-\gamma(\eta))d\gamma(\eta).\label{eqn:EnergyIntegralDef-1a}
\end{equation*}
and
\begin{equation}
\mathcal{H}_X^{\phi}(\gamma)=\frac{1}{2}\int_0^1\psi^{\gamma}(\gamma(\xi))\cdot
d\gamma(\xi) = \frac{1}{2}\int_0^1 \int_0^1 \phi(\gamma(\xi)-\gamma(\eta))
(d\gamma(\xi)\cdot d\gamma(\eta)).
\label{eqn:EnergyIntegralDef-1}
\end{equation}
The quantity
$\mathcal{H}_X^{\phi}(\gamma)$ is called the energy of path $\gamma$. We will omit $\phi$ below.
\end{definition}
\begin{remark}\label{rem:CorrectEnergyDef}
Definition
\eqref{eqn:EnergyIntegralDef-1}
is well posed. Indeed, by Lemma \ref{lem:ParametricRoughPath}
$\psi^{\gamma}\in C^2(\Rnu^3,\Rnu^3)$ and, therefore, it
follows by Lemma \ref{lem:RegularMap} that
$\psi^{\gamma} \circ\gamma\in \mathcal{D}_X$.
Moreover, if $\nu>{1}/{2}$ and $\gamma\in C^1(\mathbb{T},\Rnu^3)$, 
then by Remark \ref{rem:SmoothCase} the line integrals in the
definition of the energy are understood in the sense of Young.
\end{remark}
\begin{lemma}\label{lem:EnergyProperties}
Assume that $\phi\in C^4(\Rnu^3,\Rnu)$. Then there exists a constant
$C=C(\nu,\mathbb{X})$ such that for all $\gamma\in \mathcal{D}_X$
\begin{equation}
\vert \mathcal{H}_X(\gamma)\vert \leq
C\Vert \phi\Vert_{C^4}\Vert \gamma\Vert_{\mathcal{D}_X}^4(1+\Vert \gamma\Vert_{\mathcal{D}_X})^2.\label{eqn:EnergyEstimAbove}
\end{equation}
Furthermore, for any $R>0$ there exists $C=C(R)$ such that for any
$\gamma\in\mathcal{D}_X$,
$\tilde{\gamma}\in\mathcal{D}_{\tilde{X}}$ satisfying
$$
\Vert \gamma\Vert_{\mathcal{D}_X}\leq R,
\Vert \tilde{\gamma}\Vert_{\mathcal{D}_{\tilde{X}}}\leq R,
$$
$$
C_{X,\tilde{X}}=\Vert X\Vert_{C^{\nu}}+\Vert \tilde{X}\Vert_{C^{\nu}}+\Vert \mathbb{X}^2\Vert_{C_2^{2\nu}}+\Vert \mathbb{\tilde{X}}^2\Vert_{C_2^{2\nu}}<R
$$
we have
\begin{equation}
\vert \mathcal{H}_X(\gamma)-\mathcal{H}_{\tilde{X}}(\tilde{\gamma})\vert \leq
C(R)D(\gamma,\tilde{\gamma}).\label{eqn:LocLipshitz-2}
\end{equation}
In particular the map $\mathcal{H}_X:\mathcal{D}_X\to\Rnu$ is 
Lipshitz on balls,  i.e. for any $R>0$ there exists $C=C(R)$ such that for
any $\gamma,\tilde{\gamma}\in\mathcal{D}_X$: 
$\Vert \gamma\Vert_{\mathcal{D}_X}\leq R$,  $\Vert \tilde{\gamma}\Vert_{\mathcal{D}_X}\leq R$
we have
\begin{equation}
\vert \mathcal{H}_X(\gamma)-\mathcal{H}_X(\tilde{\gamma})\vert \leq
C(R)\Vert \gamma-\tilde{\gamma}\Vert_{\mathcal{D}_X}^*.\label{eqn:LocLipshitz}
\end{equation}
\end{lemma}
\begin{proof}
First we will show inequality \eqref{eqn:EnergyEstimAbove}. By
representation \eqref{eqn:IntegralDef-2} we have
\begin{align}
\mathcal{H}_X(\gamma)&=\frac{1}{2}(\psi^{\gamma}(\gamma(0))(\gamma(1)-\gamma(0)))
+\left[\nabla\psi^{\gamma}(\gamma(0))\gamma^\prime(0)\right]\gamma^\prime(0)\mathbb{X}^2(1,0)+Q(0,1)\nonumber\\
&=I+II+III . 
\end{align}
Since $\gamma(1)=\gamma(0)$ we infer that $I=0$.
Concerning the second term by Lemma \ref{lem:ParametricRoughPath}
we have the following estimate
\begin{align}
\vert II\vert &\leq
\Vert \mathbb{X}^2\Vert_{C_2^{2\nu}}\Vert \nabla\psi^\gamma\Vert_{L^{\infty}}\Vert \gamma^\prime\Vert_{L^{\infty}}^2
\leq
C(\nu,\mathbb{X})\Vert \phi\Vert_{C^3}\Vert \gamma\Vert_{\mathcal{D}_X}^4(1+\Vert \gamma\Vert_{\mathcal{D}_X}).\label{eqn:EnergyEstimAbove-1}
\end{align}
For third term,  from inequality \eqref{eqn:IntegralDef-3}  we infer that 
\begin{equation}
\vert III\vert \leq \Vert Q\Vert_{C_2^{3\nu}}\leq
C(\nu,\mathbb{X})\Vert \psi^\gamma(\gamma)\Vert_{\mathcal{D}_X}\Vert \gamma\Vert_{\mathcal{D}_X}.\label{eqn:EnergyEstimAbove-2}
\end{equation}
Then by Lemmata \ref{lem:RegularMap} and
\ref{lem:ParametricRoughPath} we have
\begin{align}
\Vert \psi^{\gamma}(\gamma)\Vert_{\mathcal{D}_X}&\leq
C(\nu,\mathbb{X})\Vert \psi^{\gamma}\Vert_{C^2}\Vert \gamma\Vert_{\mathcal{D}_X}(1+\Vert \gamma\Vert_{\mathcal{D}_X})\nonumber\\
&\leq C(\nu,\mathbb{X})\Vert \phi\Vert_{C^4}\Vert \gamma\Vert_{\mathcal{D}_X}^3(1+\Vert \gamma\Vert_{\mathcal{D}_X})^2 . \label{eqn:EnergyEstimAbove-3}
\end{align}
Combining inequalities \eqref{eqn:EnergyEstimAbove-1},
\eqref{eqn:EnergyEstimAbove-2} and~\eqref{eqn:EnergyEstimAbove-3}
we get inequality~\eqref{eqn:EnergyEstimAbove}.
To prove  inequality~\eqref{eqn:LocLipshitz-2} we start by formula~\eqref{eqn:IntegralDef-3} to get
\begin{align}\label{eqn:EnergyLocLip-1}
\mathcal{H}_X(\gamma)-\mathcal{H}_{\tilde{X}}(\tilde{\gamma})&=\frac{1}{2}
\Big[(\nabla\psi^{\gamma}(\gamma(0))\gamma^\prime(0)\gamma^\prime(0)-\nabla\psi^{\tilde{\gamma}}(\tilde{\gamma}(0))\tilde{\gamma}^\prime(0)\tilde{\gamma}^\prime(0))
\mathbb{X}^2(1,0)\nonumber\\
&+\nabla\psi^{\tilde{\gamma}}(\tilde{\gamma}(0))\tilde{\gamma}^\prime(0)\tilde{\gamma}^\prime(0)(\mathbb{X}^2(1,0)-\mathbb{\tilde{X}}^2(1,0))
+Q(0,1)-\tilde{Q}(0,1)\Big]\!
\nonumber\\
 &=:\!I\! +\! II\! +\! III
\end{align}
The first term in~\eqref{eqn:EnergyLocLip-1} can be represented as
follows
\begin{eqnarray}
I &=&
(\nabla\psi^{\gamma}(\gamma(0))\gamma^\prime(0)\gamma^\prime(0)-\nabla\psi^{\tilde{\gamma}}(\tilde{\gamma}(0))\tilde{\gamma}^\prime(0)\tilde{\gamma}^\prime(0))
\mathbb{X}^2(1,0)\nonumber\\
 &=& \left[(\nabla\psi^{\gamma}(\gamma(0))-\nabla\psi^{\tilde{\gamma}}(\tilde{\gamma}(0)))\gamma^\prime(0)\gamma^\prime(0)\right.\nonumber\\
 &+& \nabla\psi^{\tilde{\gamma}}(\tilde{\gamma}(0))(\gamma^\prime(0)-\tilde{\gamma}(0))\gamma^\prime(0)\nonumber\\
 &+& \left.\nabla\psi^{\tilde{\gamma}}(\tilde{\gamma}(0))\tilde{\gamma}(0)(\gamma^\prime(0)-\tilde{\gamma}(0))\right]\mathbb{X}^2(1,0)=A+B+C ,\label{eqn:EnergyLocLip-2}
\end{eqnarray}
and the first term in~\eqref{eqn:EnergyLocLip-2} can be estimated as
follows
\begin{eqnarray}\label{eqn:EnergyLocLip-3}
\vert A\vert &=&\vert (\nabla\psi^{\gamma}(\gamma(0))-\nabla\psi^{\tilde{\gamma}}(\tilde{\gamma}(0)))\gamma^\prime(0)\gamma^\prime(0)\mathbb{X}^2(1,0)\vert \\
&\leq&
\Vert \mathbb{X}^2\Vert_{C_2^{2\nu}}\Vert \gamma\Vert_{\mathcal{D}_X}^2(\vert \nabla\psi^{\gamma}
(\gamma(0))\nonumber\\
&-&\nabla\psi^{\gamma}(\tilde{\gamma}(0))\vert +\vert \nabla\psi^{\gamma}(\tilde{\gamma}(0))-\nabla\psi^{\tilde{\gamma}}(\tilde{\gamma}(0))\vert )\nonumber\\
&\leq &\Vert \mathbb{X}^2\Vert_{C_2^{2\nu}}\Vert \gamma\Vert_{\mathcal{D}_X}^2(\Vert \psi^{\gamma}\Vert_{C^2}\vert \gamma(0)-\tilde{\gamma}(0)\vert
\nonumber\\
&+&C_X^4\Vert \phi\Vert_{C^4}(1+\Vert \gamma\Vert_{\mathcal{D}_X}+\Vert \tilde{\gamma}\Vert_{\mathcal{D}_{\tilde{X}}})^3 D( \gamma,\tilde{\gamma})
\nonumber\\
&\leq&
KC_X^4\Vert \phi\Vert_{C^4}(1+\Vert \gamma\Vert_{\mathcal{D}_X}+\Vert \tilde{\gamma}\Vert_{\mathcal{D}_{\tilde{X}}})^3 D( \gamma,\tilde{\gamma}).
\nonumber
\end{eqnarray}
Here the second inequality follows from inequality
\eqref{eqn:ParametricRoughPath-2general} and the third one from
inequality \eqref{eqn:ParametricRoughPath-1}. For second term in~\eqref{eqn:EnergyLocLip-2} we
have by inequality \eqref{eqn:ParametricRoughPath-1}
\begin{eqnarray}
\vert B\vert &\leq&
C\Vert \mathbb{X}^2\Vert_{C_2^{2\nu}}\Vert \gamma\Vert_{\mathcal{D}_X}\Vert \phi\Vert_{C^3}\Vert \tilde{\gamma}\Vert_{\mathcal{D}_X}^2
(1+\Vert \tilde{\gamma}\Vert_{\mathcal{D}_X})D( \gamma,\tilde{\gamma})\nonumber
\\
&\leq& CC_X(1+\Vert \gamma\Vert_{\mathcal{D}_X}+\Vert \tilde{\gamma}\Vert_{\mathcal{D}_{\tilde{X}}})^3 D( \gamma,\tilde{\gamma}).
\label{eqn:EnergyLocLip-4}
\end{eqnarray}
Similarly, we have for third term
\begin{eqnarray}
\vert C\vert \leq
C(\nu,\mathbb{X},\Vert \gamma\Vert_{\mathcal{D}_X},\Vert \tilde{\gamma}\Vert_{\mathcal{D}_X})D( \gamma,\tilde{\gamma}).\label{eqn:EnergyLocLip-5}
\end{eqnarray}
Going back to the term $II$ in~\eqref{eqn:EnergyLocLip-1} we observe that it can be estimated as
follows
\begin{align*}
\vert II\vert &\leq
\Vert \nabla\psi^{\tilde{\gamma}}\Vert_{L^{\infty}}\Vert \tilde{\gamma}\Vert_{\mathcal{D}_X}^2D( \gamma,\tilde{\gamma})
\leq
C_X^3\Vert \phi\Vert_{C^3}(1+\Vert \gamma\Vert_{\mathcal{D}_X}+\Vert \tilde{\gamma}\Vert_{\mathcal{D}_{\tilde{X}}})^3D( \gamma,\tilde{\gamma}).
\end{align*}
Thus it remains to estimate third term of equality
\eqref{eqn:EnergyLocLip-1}. By inequality~\eqref{eqn:IntegralDef-5} we have
\begin{align*}
\vert Q(0,1)-\tilde{Q}(0,1)\vert &\leq \Vert Q-\tilde{Q}\Vert_{C_2^{3\nu}}\\
&\leq
C_X\Big[(\Vert \psi^{\tilde{\gamma}}(\tilde{\gamma})\Vert_{\mathcal{D}_{\tilde{X}}}+
\Vert \psi^{\gamma}(\gamma)\Vert_{\mathcal{D}_X})D( \gamma,\tilde{\gamma})\\
&+(\Vert \tilde{\gamma}\Vert_{\mathcal{D}_{\tilde{X}}}+\Vert \tilde{\gamma}\Vert_{\mathcal{D}_{\tilde{X}}})D( \psi^{\tilde{\gamma}}(\tilde{\gamma}), \psi^{\gamma}(\gamma))\\
&+(\Vert \psi^{\tilde{\gamma}}(\tilde{\gamma})\Vert_{\mathcal{D}_{\tilde{X}}}+
\Vert \psi^{\gamma}(\gamma)\Vert_{\mathcal{D}_X})(\Vert \tilde{\gamma}\Vert_{\mathcal{D}_{\tilde{X}}}+\Vert \tilde{\gamma}\Vert_{\mathcal{D}_{\tilde{X}}})\\
&\times(\Vert X-\tilde{X}\Vert_{C^{\nu}}+\Vert \mathbb{X}^2-\mathbb{\tilde{X}}^2\Vert_{C_2^{2\nu}})\Big]
\end{align*}
By
inequality~\eqref{eqn:ParametricRoughPath-1}, the term $\Vert \psi^{\tilde{\gamma}}(\tilde{\gamma})\Vert_{\mathcal{D}_X}$ is
bounded by the constant
$C=C(\nu,\mathbb{X},\Vert \tilde{\gamma}\Vert_{\mathcal{D}_X})$. Therefore, to prove
estimate~\eqref{eqn:LocLipshitz} it is enough to show that there
exists a constant $C=C(\nu,\mathbb{X},R)$ such that for
$\gamma,\tilde{\gamma}\in\mathcal{D}_X$ with
$\Vert \gamma\Vert_{\mathcal{D}_X},\Vert \tilde{\gamma}\Vert_{\mathcal{D}_X}\leq R$
\begin{equation}
D( \psi^{\tilde{\gamma}}(\tilde{\gamma}), \psi^{\gamma}(\gamma)) \leq
CD( \gamma,\tilde{\gamma}).\label{eqn:EnergyLocLip-6}
\end{equation}
By the triangle inequality we have
\begin{align}
D( \psi^{\tilde{\gamma}}(\tilde{\gamma}), \psi^{\gamma}(\gamma))
&\leq D( \psi^{\tilde{\gamma}}(\tilde{\gamma}), \psi^{\tilde{\gamma}}(\gamma)) + D( \psi^{\tilde{\gamma}}({\gamma}), \psi^{\gamma}(\gamma))
\nonumber\\
&=I+II.\label{eqn:EnergyLocLip-7}
\end{align}
The first term can be estimated by using inequality
\eqref{eqn:RegularMap-4} as follows
\begin{equation}
\vert I\vert \leq
KC_X^3\Vert \psi^{\tilde{\gamma}}\Vert_{C^3}(1+\Vert \tilde{\gamma}\Vert_{\mathcal{D}_{\tilde{X}}}+
\Vert \gamma\Vert_{\mathcal{D}_X})^2 D( \gamma,\tilde{\gamma}).\label{eqn:EnergyLocLip-8}
\end{equation}
By inequality \eqref{eqn:ParametricRoughPath-1} we have
\begin{equation}
\Vert \psi^{\tilde{\gamma}}\Vert_{C^3}\leq
C\Vert \phi\Vert_{C^5}\Vert \tilde{\gamma}\Vert_{\mathcal{D}_X}^2(1+\Vert \tilde{\gamma}\Vert_{\mathcal{D}_X})\label{eqn:EnergyLocLip-8'}
\end{equation}
Combining inequalities \eqref{eqn:EnergyLocLip-8} and
\eqref{eqn:EnergyLocLip-8'} we obtain the necessary estimate for $I$. It
remains to find an estimate for the term $II$. By inequalities
\eqref{eqn:RegularMap-3} and \eqref{eqn:ParametricRoughPath-2} we
have
\begin{align}
II&=D( \psi^{\tilde{\gamma}}({\gamma}), \psi^{\gamma}(\gamma))
\leq
(1+\Vert X\Vert_{\nu})\Vert \psi^{\tilde{\gamma}}(\gamma)-\psi^{\gamma}(\gamma)\Vert_{\mathcal{D}_X}\nonumber\\
&\leq K\Vert \nabla\psi^{\tilde{\gamma}}-\nabla\psi^{\gamma}\Vert_{C^1}
\Vert \gamma\Vert_{\mathcal{D}_X}(1+\Vert \gamma\Vert_{\mathcal{D}_X})(1+\Vert X\Vert_{\nu})^3
\nonumber\\
&\leq K\Vert \phi\Vert_{C^5}C_X^7(1+\Vert \gamma\Vert_{\mathcal{D}_X}+\Vert \tilde{\gamma}\Vert_{\mathcal{D}_{\tilde{X}}})^5
D( \gamma,\tilde{\gamma}).\label{eqn:EnergyLocLip-9}
\end{align}
Hence the inequality \eqref{eqn:LocLipshitz-2} follows. Finally the bound~\eqref{eqn:LocLipshitz} is a consequence of eq.~\eqref{eqn:LocLipshitz-2}.
\end{proof}
Let us recall the definition \eqref{eqn:RandomFilam-3} of the vector field $u^\gamma$ generated by a controlled path $\gamma$:
\begin{equation}
u^\gamma(x)=\int_0^1\nabla\phi(x-\gamma(\xi))\times d \gamma(\xi),\qquad \gamma\in
\mathcal{D}_X.\label{eqn:RandomFilam-3'}
\end{equation}
Now we will show that if the energy functional of $\gamma$ is bounded then the
associated velocity field  is a smooth function. We have
\begin{lemma}
\label{lem:VelocitySmoothness}
Assume that $\phi\in C^4(\Rnu^3,\Rnu)$ and Hypothesis~\ref{ass:RoughPathApproximability} holds. Then the energy function
$\mathcal{H}_X:\mathcal{D}_X\to\Rnu$ is  continuous. Furthermore if $\gamma$ is a geometric rough path then
\begin{equation}
\label{eq:energy-representation}
\mathcal{H}_X(\gamma) =\frac{1}{(2\pi)^3}\int_{\Rnu^3}\hat{\phi}(k)\Big\vert \int_0^1 e^{i(k,\gamma(\xi))}d\gamma(\xi)\Big\vert^2dk \ge 0.
\end{equation}
Moreover,  if in addition the integral
$\int_{\Rnu^3}\vert {k}\vert^{2(1+n)}\hat{\phi}({k})d{k}$ is
finite and $\phi\in C^{n+4}(\Rnu^3,\Rnu^3)$, then  for any $n\in \mathbb{N}0$, we have following bound
\begin{equation}
\Vert \nabla^n u^{\gamma}\Vert^2_{L^{\infty}}\leq
\frac{1}{(2\pi)^{3}}\left[\int_{\Rnu^3}\vert {k}\vert^{2(1+n)}\hat{\phi}({k})d{k}\right]\mathcal{H}_X(\gamma),\qquad
\gamma\in\mathcal{D}_X.\label{eqn:VelocitySmoothness}
\end{equation}
provided that 
\end{lemma}
\begin{proof}
For a smooth curve $\gamma$ Lemma~\ref{lem:VelocitySmoothness} has been proved in
\cite[Lemma~3]{[Bers+Gub_2007]}. In the general case, when
$\gamma\in\mathcal{D}_X$, it is enough to notice that both sides
of eq.~\eqref{eq:energy-representation} and of inequality~\eqref{eqn:VelocitySmoothness} are locally Lipshitz
and therefore, continuous w.r.t. distance
$D(Y,\tilde{Y})$, $Y\in\mathcal{D}_X$,
$\tilde{Y}\in\mathcal{D}_{\tilde{X}}$ defiend in \eqref{eqn-distance}. Indeed the continuity of
$\mathcal{H}_X$  and continuity of $\Vert \nabla^n
u^{\gamma}\Vert_{L^{\infty}}$ readily follows from Lemma
\ref{lem:ParametricRoughPath}.
\end{proof}

Now we are going to show that energy is a local integral of motion
for problem~\eqref{eqn:RandomFilam-1}.
\begin{lemma}\label{lem:EnergyMotionIntegral}
Let $\gamma_0$ a geometric $\nu$-rough path and $\gamma\in\mathcal{D}_{\mathbb{\gamma}_0,T_0}$ be a local solution of
problem~\eqref{eqn:RandomFilam-1} (such a solution exists by Theorem \ref{thm:LocalExistUniq}). Then
$$
\frac{d\mathcal{H}_{\gamma_0}(\gamma(s))}{ds}=0,\qquad s\in[0,T_0).
$$
\end{lemma}
\begin{proof}
Since $\gamma(0)=\gamma_0\in \mathcal{D}_{\gamma_0}$ is a geometric rough path (we will denote its area component by $\Gamma_0$)
there exist sequence $\{\gamma_0^n\}_{n=1}^{\infty}$ of piecewise smooth closed curves in $\Rnu^3$ such that if we denote by $(\gamma_0^n,\Gamma_0^n)$ their canonical lift to the space of $\nu$-rough paths we have
\[
\Vert \gamma_0^n-\gamma_0\Vert_{C^{\nu}}+\Vert \Gamma^n_0-{\Gamma}_0\Vert_{C_2^{2\nu}}\to 0,\qquad n\to\infty.
\]
Now observe that $\gamma_0^n\in \mathcal{D}_{\gamma_0^n}$, $\gamma_0\in \mathcal{D}_{\gamma_0}$ since we can take
$(\gamma_0^n)^\prime=(\gamma_0)^\prime=1$ and $R^{\gamma_0^n}=R^{\gamma_0}=0$.
Hence we deduce that
$$
D( \gamma_0^n,\gamma_0)\to 0,\qquad n\to\infty.
$$
Denote by $\gamma^n\in C([0,\infty),\mathbf{H}^1(\mathbb{T},\Rnu^3))$ the
global solution of problem~\eqref{eqn:RandomFilam-1} with initial
condition $\gamma_0^n$. Existence of such solution has been proved in
Theorem $2$ of \cite{[Bers+Gub_2007]}. Moreover note that for smooth functions $\gamma^n$ controlled by the rough path $(\gamma_0^n,\Gamma_0^n)$ the integral defined via rough paths coincide with the standard Lebesgue integral and the solution of the Cauchy problem in $C([0,\infty),\mathbf{H}^1(\mathbb{T},\Rnu^3))$ with the rough solution whose local existence is stated in Theorem~\ref{thm:LocalExistUniq}.   The locally Lipshitz dependence of the rough solution on  the initial data  stated in Theorem~\ref{thm:LocalExistUniq} implies that
$$
\lim_{n\to \infty} \supl_{t\in[0,T_0]}D( \gamma^n(t),\gamma(t)) \lesssim  \lim_{n\to \infty}  D( \gamma^n(0),\gamma(0))  =  0
$$
for a stricly positive $T_0$ which depends only on the rough path norm of $\gamma_0$.
Therefore, by the continuity of the energy functional $\mathcal{H}_{\gamma_0}$ we
have
\begin{equation}
\mathcal{H}_{\gamma_0}(\gamma(s))=\liml_{n\to\infty}\mathcal{H}_{\gamma_0^n}(\gamma^n(s)),\qquad s\in[0,T_0].\label{eqn:lemEMIaux-1}
\end{equation}
Furthermore, by Lemma $2$ of \cite{[Bers+Gub_2007]}, we
have
\begin{equation}
\mathcal{H}_{\gamma_0^n}(\gamma^n(s))=\mathcal{H}_{\gamma_0^n}(\gamma_0^n),\qquad s\in[0,T_0].\label{eqn:lemEMIaux-2}
\end{equation}
As a result, combining inequalities \eqref{eqn:lemEMIaux-1} and
\eqref{eqn:lemEMIaux-2} we get the statement of the Lemma.
\end{proof}
Now we are ready to prove Theorem \ref{thm:GlobalExistUniq}.
\begin{proof}[Proof of Theorem \ref{thm:GlobalExistUniq}]
According to Theorem \ref{thm:LocalExistUniq} there exists a unique
local solution of problem~\eqref{eqn:RandomFilam-1}. Then, we can
find $T^*>0$ and a  unique maximal local solution 
$\gamma:[0,T^*)\to \mathcal{D}_{\gamma_0}$ which then satisfies
\begin{equation}
\liml_{t\nearrow T^*}\Vert \gamma(t)\Vert_{\mathcal{D}_{\gamma_0}}=\infty.\label{eqn:thmGEUaux-1}
\end{equation}
We need
to show that $T^*=\infty$. Therefore, it is enough to prove
$$
\supl_{t\in[0,T^*)}\Vert \gamma(t)\Vert_{\mathcal{D}_{\gamma_0}}<\infty.
$$
Indeed, by contradiction with \eqref{eqn:thmGEUaux-1}, the result
will follow. In the rest of the proof we show such estimate.
Notice that we will have
\begin{equation}
{\mathcal{H}_{\gamma_0}(\gamma(s))}={\mathcal{H}_{\gamma_0}(\gamma_0)},\qquad s\in[0,T^*),\label{eqn:EnergyMotionIntegral-1}
\end{equation}
and
recall that
\begin{equation}
\gamma(t)=\gamma_0+\int_0^tu^{\gamma(s)}(\gamma(s))ds.\label{eqn:RandomFilam-12}
\end{equation}
Firstly we have
\begin{align}
\Vert \gamma(t)\Vert_{L^{\infty}}&\leq\Vert \gamma_0\Vert_{L^{\infty}}+\int_0^t\Vert u^{\gamma(s)}\Vert_{L^{\infty}}ds
\leq\Vert \gamma_0\Vert_{L^{\infty}}+C\int_0^t\mathcal{H}_{\gamma_0}^{1/2}(\gamma(s))ds\nonumber\\
&\leq\Vert \gamma_0\Vert_{L^{\infty}}+C\mathcal{H}_{\gamma_0}^{1/2}(\gamma_0)t,\qquad t\in[0,T^*).\label{eqn:GEUEstim-1}
\end{align}
It follows from \eqref{eqn:RandomFilam-12} that
\begin{equation}
\gamma^\prime(t)=\gamma_0^\prime+\int_0^t\nabla
u^{\gamma(s)}(\gamma(s))\gamma^\prime(s)ds,t\in[0,T^*).\label{eqn:RandomFilam-12'}
\end{equation}
Therefore, by Lemmata \ref{lem:EnergyMotionIntegral} and
\ref{lem:VelocitySmoothness} we have
\begin{align}
\Vert \gamma^\prime(t)\Vert_{L^{\infty}}&\leq\Vert \gamma^\prime_0\Vert_{L^{\infty}}+\int_0^t\Vert \nabla
u^{\gamma(s)}\Vert_{L^{\infty}}\Vert \gamma^\prime(s)\Vert_{L^{\infty}}ds
\leq\Vert \gamma^\prime_0\Vert_{L^{\infty}}+\int_0^tC\mathcal{H}_{\gamma_0}^{1/2}(\gamma(s))\Vert \gamma^\prime(s)\Vert_{L^{\infty}}ds\nonumber\\
&=\Vert \gamma^\prime_0\Vert_{L^{\infty}}+\int_0^t C\mathcal{H}_{\gamma_0}^{1/2}(\gamma_0)\Vert \gamma^\prime(s)\Vert_{L^{\infty}}ds,\qquad t\in[0,T^*).
\end{align}
Then by the Gronwall Lemma we infer our second estimate
\begin{equation}
\Vert \gamma^\prime(t)\Vert_{L^{\infty}}\leq\Vert \gamma^\prime_0\Vert_{L^{\infty}}e^{C\mathcal{H}_{\gamma_0}^{1/2}(\gamma_0)t},\qquad t\in[0,T^*).\label{eqn:GEUEstim-2}
\end{equation}
We will need one more auxiliary estimate. We have
\begin{align}
\Vert \gamma(t)\Vert_{C^{\nu}}&\leq\Vert \gamma_0\Vert_{C^{\nu}}+\int_0^t\Vert
u^{\gamma(s)}(\gamma(s))\Vert_{C^{\nu}}\Vert ds
\leq\Vert \gamma_0\Vert_{C^{\nu}}+\int_0^t\Vert \nabla
u^{\gamma(s)}\Vert_{L^{\infty}}\Vert \gamma(s)\Vert_{C^{\nu}}ds\nonumber\\
&\leq
\Vert \gamma_0\Vert_{C^{\nu}}+\int_0^tC\mathcal{H}_{\gamma_0}^{1/2}(\gamma(s))\Vert \gamma(s)\Vert_{C^{\nu}}ds\nonumber\\
&=\Vert \gamma_0\Vert_{C^{\nu}}+\int_0^tC\mathcal{H}_{\gamma_0}^{1/2}(\gamma_0)\Vert \gamma(s)\Vert_{C^{\nu}}ds,\qquad t\in[0,T^*).
\end{align}
Thus, by the Gronwall Lemma we get
\begin{equation}
\Vert \gamma(t)\Vert_{C^{\nu}}\leq\Vert \gamma_0\Vert_{C^{\nu}}e^{C\mathcal{H}_{\gamma_0}^{1/2}(\gamma_0)t},t\in[0,T^*).\label{eqn:GEUEstim-aux}
\end{equation}
Now we can estimate $C^{\nu}$ norm of $\gamma^\prime$. We have
\begin{align}
\Vert \gamma^\prime(t)\Vert_{C^{\nu}}&\leq\Vert \gamma^\prime_0\Vert_{C^{\nu}}+\int_0^t\Vert \nabla
u^{\gamma(s)}(\gamma(s))\gamma^\prime(s)\Vert_{C^{\nu}}ds\nonumber\\
&\leq \Vert \gamma^\prime_0\Vert_{C^{\nu}}+\int_0^t(\Vert \nabla
u^{\gamma(s)}\Vert_{L^{\infty}}\Vert \gamma^\prime(s)\Vert_{C^{\nu}}+\Vert \gamma^\prime(s)\Vert_{L^{\infty}}\Vert \nabla
u^{\gamma(s)}(\gamma(s))\Vert_{C^{\nu}})ds\nonumber\\
&\leq\Vert \gamma^\prime_0\Vert_{C^{\nu}}+\int_0^t(\Vert \nabla
u^{\gamma(s)}\Vert_{L^{\infty}}\Vert \gamma^\prime(s)\Vert_{C^{\nu}}+\Vert \gamma^\prime(s)\Vert_{L^{\infty}}\Vert \nabla^2
u^{\gamma(s)}\Vert_{L^{\infty}}\Vert \gamma(s)\Vert_{C^{\nu}})ds\nonumber\\
&\leq\Vert \gamma^\prime_0\Vert_{C^{\nu}}\nonumber\\
&+\int_0^t(C\mathcal{H}_{\gamma_0}^{1/2}(\gamma_0)(\Vert \gamma^\prime(s)\Vert_{C^{\nu}}+\Vert \gamma^\prime_0\Vert_{L^{\infty}}\Vert \gamma_0\Vert_{C^{\nu}}e^{C\mathcal{H}_{\gamma_0}^{1/2}(\gamma_0)s}))ds,\qquad t\in[0,T^*),
\end{align}
where last inequality follows from Lemmata
\ref{lem:EnergyMotionIntegral} and \ref{lem:VelocitySmoothness}.
Then by the Gronwall Lemma we get the third estimate
\begin{equation}
\Vert \gamma^\prime(t)\Vert_{C^{\nu}}\leq(\Vert \gamma^\prime_0\Vert_{C^{\nu}}+\Vert \gamma^\prime_0\Vert_{L^{\infty}}\Vert \gamma_0\Vert_{C^{\nu}})e^{C\mathcal{H}_{\gamma_0}^{1/2}(\gamma_0)t},\qquad t\in[0,T^*).\label{eqn:GEUEstim-3}
\end{equation}
It remains to find an estimate for $\Vert R^{\gamma(t)}\Vert_{2\nu}$. We
have
\begin{equation}
R^{\gamma(t)}=R^{\gamma_0}+\int_0^tR^{u^{\gamma(s)}(\gamma(s))}ds,\qquad t\in[0,T^*).
\end{equation}
By identity \eqref{eqn:RegularMap-2} we have for $s\in[0,T^*)$
\begin{align}
R^{u^{\gamma(s)}(\gamma(s))}(\xi,\eta)&=\nabla
u^{\gamma(s)}(\gamma(s,\xi))R^{\gamma(s)}(\xi,\eta)+
\suml_k(\gamma^k(s,\eta)-\gamma^k(s,\xi))\times\nonumber\\
&\int_0^1\left[\frac{\partial u^{\gamma(s)}}{\partial
x_k}(\gamma(s,\xi)+r(\gamma(s,\eta)-\gamma(s,\xi)))-\frac{\partial
u^{\gamma(s)}}{\partial
x_k}(\gamma(s,\xi))\right]dr.\label{eqn:GEUEstim-aux3}
\end{align}
Therefore,
\begin{equation}
\Vert R^{u^{\gamma(s)}(\gamma(s))}\Vert_{{C}_2^{2\nu}}\leq \Vert \nabla
u^{\gamma(s)}\Vert_{L^{\infty}}\Vert R^{\gamma(s)}\Vert_{{C}_2^{2\nu}}+\frac{1}{2}\Vert \gamma(s)\Vert_{C^{\nu}}^2\Vert \nabla^2
u^{\gamma(s)}\Vert_{L^{\infty}},\quad s\in[0,T^*).\label{eqn:GEUEstim-aux4}
\end{equation}
Thus, by inequalities \eqref{eqn:GEUEstim-aux4} and
\eqref{eqn:GEUEstim-aux} we have for $t\in[0,T^*)$
\begin{align}
\Vert R^{\gamma(t)}\Vert_{{C}_2^{2\nu}}&\leq\Vert R^{\gamma_0}\Vert_{{C}_2^{2\nu}}+\int_0^t(\Vert \nabla
u^{\gamma(s)}\Vert_{L^{\infty}}\Vert R^{\gamma(s)}\Vert_{{C}_2^{2\nu}}+\frac{1}{2}\Vert \gamma(s)\Vert_{C^{\nu}}^2\Vert \nabla^2
u^{\gamma(s)}\Vert_{L^{\infty}})ds\nonumber\\
&\leq\Vert R^{\gamma_0}\Vert_{{C}_2^{2\nu}}+\int_0^t(\Vert \nabla
u^{\gamma(s)}\Vert_{L^{\infty}}\Vert R^{\gamma(s)}\Vert_{{C}_2^{2\nu}}+\Vert \gamma_0\Vert_{C^{\nu}}e^{C\mathcal{H}_{\gamma_0}^{\half}(\gamma_0)t}\Vert \nabla^2
u^{\gamma(s)}\Vert_{L^{\infty}})ds\nonumber\\
&\leq\Vert R^{\gamma_0}\Vert_{{C}_2^{2\nu}}+C(\Vert \gamma_0\Vert_{C^{\nu}},\mathcal{H}_{\gamma_0}^{1/2}(\gamma_0))e^{C\mathcal{H}_{\gamma_0}^{1/2}(\gamma_0)t}\nonumber\\
&+\int_0^tC\mathcal{H}_{\gamma_0}^{1/2}(\gamma_0)\Vert R^{\gamma(s)}\Vert_{{C}_2^{2\nu}}ds,\label{eqn:GEUEstim-aux5}
\end{align}
where in the last inequality we used Lemmata
\ref{lem:EnergyMotionIntegral} and \ref{lem:VelocitySmoothness}.
Hence, by the Gronwall Lemma we get
\begin{equation}
\Vert R^{\gamma(t)}\Vert_{{C}_2^{2\nu}}\leq(\Vert R^{\gamma_0}\Vert_{{C}_2^{2\nu}}+C(\Vert \gamma_0\Vert_{C^{\nu}},\mathcal{H}_{\gamma_0}^{1/2}(\gamma_0))e^{C\mathcal{H}_{\gamma_0}^{1/2}(\gamma_0)t}))
e^{C\mathcal{H}_{\gamma_0}^{1/2}(\gamma_0)t},\qquad t\in[0,T^*),\label{eqn:GEUEstim-4}
\end{equation}
and combining estimates \eqref{eqn:GEUEstim-1},
\eqref{eqn:GEUEstim-2}, \eqref{eqn:GEUEstim-3}, and
\eqref{eqn:GEUEstim-4} we prove following a-priori estimate
\begin{equation}
\Vert \gamma(t)\Vert_{\mathcal{D}_{\gamma_0}}\leq
K(1+\mathcal{H}_{\gamma_0}^{1/2}(\gamma_0))(1+\Vert \gamma_0\Vert_{\mathcal{D}_{\gamma_0}})\Vert \gamma_0\Vert_{\mathcal{D}_{\gamma_0}}e^{C\mathcal{H}_{\gamma_0}^{1/2}(\gamma_0)t},\quad t\in[0,T^*),\label{eqn:GEUEstim-5}
\end{equation}
and the result follows.
\end{proof}


\begin{bibdiv}
\begin{biblist}

\bib{arnold_topological_1998}{book}{	
    edition = {Corrected},
	title = {Topological Methods in Hydrodynamics},
	isbn = {{038794947X}},
	publisher = {Springer},
	author = {Arnold, V. I.},
	author = {Khesin, B. A.},
	month = {apr},
	year = {1998},
}


\bib{[BKM-1984]}{article}{
   author={Beale, J. T.},
   author={Kato, T.},
   author={Majda, A.},
   title={Remarks on the breakdown of smooth solutions for the $3$-D Euler
   equations},
   journal={Comm. Math. Phys.},
   volume={94},
   date={1984},
   number={1},
   pages={61--66},
   issn={0010-3616},
   review={\MR{763762 (85j:35154)}},
}

\bib{[Bell+Markus_1992]}{article}{
   author={Bell, J. B.},
   author={Marcus, D. L.},
   title={Vorticity intensification and transition to turbulence in the
   three-dimensional Euler equations},
   journal={Comm. Math. Phys.},
   volume={147},
   date={1992},
   number={2},
   pages={371--394},
   issn={0010-3616},
   review={\MR{1174419 (93c:76048)}},
}

\bib{[Bers+Bess_2002]}{article}{
   author={Berselli, L. C.},
   author={Bessaih, H.},
   title={Some results for the line vortex equation},
   journal={Nonlinearity},
   volume={15},
   date={2002},
   number={6},
   pages={1729--1746},
   issn={0951-7715},
   review={\MR{1938468 (2003m:76028)}},
   doi={10.1088/0951-7715/15/6/301},
}

\bib{[Bers+Gub_2007]}{article}{
   author={Berselli, L. C.},
   author={Gubinelli, M.},
   title={On the global evolution of vortex filaments, blobs, and small
   loops in 3D ideal flows},
   journal={Comm. Math. Phys.},
   volume={269},
   date={2007},
   number={3},
   pages={693--713},
   issn={0010-3616},
   review={\MR{2276358 (2007m:76027)}},
   doi={10.1007/s00220-006-0142-x},
}

\bib{[Bess+Gub_2005]}{article}{
   author={Bessaih, H.},
   author={Gubinelli, M.},
   author={Russo, F.},
   title={The evolution of a random vortex filament},
   journal={Ann. Probab.},
   volume={33},
   date={2005},
   number={5},
   pages={1825--1855},
   issn={0091-1798},
   review={\MR{2165581 (2006i:60069)}},
   doi={10.1214/009117905000000323},
}

\bib{[Brylinski-1992]}{book}{
   author={Brylinski, J.-L.},
   title={Loop spaces, characteristic classes and geometric quantization},
   series={Progress in Mathematics},
   volume={107},
   publisher={Birkh\"auser Boston Inc.},
   place={Boston, MA},
   date={1993},
   pages={xvi+300},
   isbn={0-8176-3644-7},
   review={\MR{1197353 (94b:57030)}},
}

\bib{Chorin}{book}{
   author={Chorin, A. J.},
   title={Vorticity and turbulence},
   series={Applied Mathematical Sciences},
   volume={103},
   publisher={Springer-Verlag},
   place={New York},
   date={1994},
   pages={viii+174},
   isbn={0-387-94197-5},
   review={\MR{1281384 (95m:76043)}},
}

\bib{[Friz-2005]}{article}{
   author={Friz, P. K.},
   title={Continuity of the It\^o-map for H\"older rough paths with
   applications to the support theorem in H\"older norm},
   conference={
      title={Probability and partial differential equations in modern
      applied mathematics},
   },
   book={
      series={IMA Vol. Math. Appl.},
      volume={140},
      publisher={Springer},
      place={New York},
   },
   date={2005},
      volume={XX},
   pages={117--135},
   review={\MR{2202036 (2007f:60070)}},
   doi={10.1007/978-0-387-29371-4\_8},
}

\bib{[FrizVictoir]}{book}{
   author={Friz, P. K.},
   author={Victoir, N. B.},
   title={Multidimensional stochastic processes as rough paths},
   series={Cambridge Studies in Advanced Mathematics},
   volume={120},
   note={Theory and applications},
   publisher={Cambridge University Press},
   place={Cambridge},
   date={2010},
   pages={xiv+656},
   isbn={978-0-521-87607-0},
   review={\MR{2604669 (2012e:60001)}},
}

\bib{Gallavotti}{book}{
   author={Gallavotti, G.},
   title={Foundations of fluid dynamics},
   series={Texts and Monographs in Physics},
   note={Translated from the Italian},
   publisher={Springer-Verlag},
   place={Berlin},
   date={2002},
   pages={xviii+513},
   isbn={3-540-41415-0},
   review={\MR{1872661 (2003e:76002)}},
}

\bib{[Gubinelli-2004]}{article}{
   author={Gubinelli, M.},
   title={Controlling rough paths},
   journal={J. Funct. Anal.},
   volume={216},
   date={2004},
   number={1},
   pages={86--140},
   issn={0022-1236},
   review={\MR{2091358 (2005k:60169)}},
   doi={10.1016/j.jfa.2004.01.002},
}

\bib{Hairer}{article}{
   author={Hairer, M.},
   title={Rough stochastic PDEs},
   journal={Comm. Pure Appl. Math.},
   volume={64},
   date={2011},
   number={11},
   pages={1547--1585},
   issn={0010-3640},
   review={\MR{2832168}},
   doi={10.1002/cpa.20383},
}


\bib{KPZ}{article}{
   author={Hairer, M.},
   title={Solving the KPZ equation},
   journal={Ann. Math.},
   note={to appear},
   date={2013},
}



\bib{Helm1858}{article}{
	title = {\"Uber Integrale der hydrodynamischen Gleichungen, welche den Wirbelbewegungen entsprechen.},
	volume = {1858},
	issn = {0075-4102, 1435-5345},
	doi = {10.1515/crll.1858.55.25},
	number = {55},
	journal = {Journal f\"ur die reine und angewandte Mathematik {(Crelles} Journal)},
	author = {Helmholtz, H.},
	year = {1858},
	pages = {25--55},
}

\bib{Kelvin1869}{article}{
	author={ Thomson (Lord Kelvin), W.},
	title={On vortex motion},
	journal={Trans. Royal Soc. Edin.},
	number={25},
	pages={217--260},
	year={1869},
}

\bib{[Lyons-1998]}{article}{
   author={Lyons, T. J.},
   title={Differential equations driven by rough signals},
   journal={Rev. Mat. Iberoamericana},
   volume={14},
   date={1998},
   number={2},
   pages={215--310},
   issn={0213-2230},
   review={\MR{1654527 (2000c:60089)}},
   doi={10.4171/RMI/240},
}

\bib{[LionsStFlour]}{book}{
   author={Lyons, T. J.},
   author={Caruana, M.},
   author={L{\'e}vy, T.},
   title={Differential equations driven by rough paths},
   series={Lecture Notes in Mathematics},
   volume={1908},
   note={Lectures from the 34th Summer School on Probability Theory held in
   Saint-Flour, July 6--24, 2004;
   With an introduction concerning the Summer School by Jean Picard},
   publisher={Springer},
   place={Berlin},
   date={2007},
   pages={xviii+109},
   isbn={978-3-540-71284-8},
   isbn={3-540-71284-4},
   review={\MR{2314753 (2009c:60156)}},
}
	
\bib{[LyonsQian-2002]}{book}{
   author={Lyons, T.},
   author={Qian, Z.},
   title={System control and rough paths},
   series={Oxford Mathematical Monographs},
   note={Oxford Science Publications},
   publisher={Oxford University Press},
   place={Oxford},
   date={2002},
   pages={x+216},
   isbn={0-19-850648-1},
   review={\MR{2036784 (2005f:93001)}},
   doi={10.1093/acprof:oso/9780198506485.001.0001},
}


\bib{[Rosenhead-1930]}{article}{
	title = {The Spread of Vorticity in the Wake Behind a Cylinder},
	volume = {127},
	issn = {1364-5021, 1471-2946},
	doi = {10.1098/rspa.1930.0078},
	number = {806},
	journal = {Proceedings of the Royal Society A: Mathematical, Physical and Engineering Sciences},
	author = {Rosenhead, L.},
	year = {1930},
	pages = {590--612},
}


\bib{[VincMeneguzzi-1991]}{article}{
	title = {The Spatial Structure and Statistical Properties of Homogeneous Turbulence},
	volume = {225},
	doi = {10.1017/S0022112091001957},
	journal = {Journal of Fluid Mechanics},
	author = {Vincent, A.},
	author = {Meneguzzi, M.},
	year = {1991},
	pages = {1--20},
}


\bib{[Young-1936]}{article}{
   author={Young, L. C.},
   title={An inequality of the H\"older type, connected with Stieltjes
   integration},
   journal={Acta Math.},
   volume={67},
   date={1936},
   number={1},
   pages={251--282},
   issn={0001-5962},
   review={\MR{1555421}},
   doi={10.1007/BF02401743},
}

\end{biblist}
\end{bibdiv}
\end{document}